\documentclass[12pt]{article}

\usepackage{amsmath}
\usepackage{amsthm}
\usepackage{amsfonts}
\usepackage{mathrsfs}
\usepackage{stmaryrd}
\usepackage{setspace}
\usepackage{fullpage}
\usepackage{amssymb}
\usepackage{breqn}
\usepackage{enumitem}
\usepackage{bbold}
\usepackage{authblk}
\usepackage{comment}
\usepackage{hyperref}
\usepackage{pgf,tikz}
\usepackage{graphicx}
\usepackage{subcaption}

\newtheorem{thm}{Theorem}[section]

\newenvironment{thmbis}[1]
{\renewcommand{\thethm} {\ref{#1}} \addtocounter{thm}{-1} 
\begin{thm}}
{\end{thm}}

\newtheorem{lemma}[thm]{Lemma}

\newtheorem{proposition}[thm]{Proposition}
\newtheorem{prop}[thm]{Proposition}

\newtheorem{cor}[thm]{Corollary}

\newtheorem{clm}[thm]{Claim}

\newcommand\ex{\ensuremath{\mathrm{ex}}}

\newcommand\cC{{\mathcal C}}

\newcommand\cF{{\mathcal F}}
\newcommand\cG{{\mathcal G}}
\newcommand\cH{{\mathcal H}}

\newcommand\cN{{\mathcal N}}

\def\lc{\left\lceil}
\def\rc{\right\rceil}

\def\lf{\left\lfloor}
\def\rf{\right\rfloor}

\newtheorem*{thm*}{Theorem}
\newtheorem*{prop*}{Proposition}
\newcommand{\ignore}[1]{}

\title{On the connected Tur\'an numbers of Berge paths and Berge cycles}
\author{
Xiamiao Zhao \thanks{\small Department of Mathematical Sciences, Tsinghua University, Beijing 100084, China. Email:
\small zxm23@mails.tsinghua.edu.cn}\,,
\hspace{0.2em}
D\'{a}niel Gerbner\thanks{\small Alfr\'ed R\'enyi Institute of Mathematics. Email:
\small \texttt{gerbner.daniel@renyi.hu}.}\,, \hspace{0.2em} 
Junpeng Zhou\thanks{\small  \textit{Corresponding author}. Department of Mathematics, Shanghai University, Shanghai 200444, P.R. China. Email:
\small \texttt{junpengzhou@shu.edu.cn}.} \thanks{\small Newtouch Center for Mathematics of Shanghai University, Shanghai 200444, P.R. China.}}

\date{}

\begin{document}

\maketitle

\begin{abstract} 
Given a graph $F$, a Berge copy of $F$ (Berge-$F$ for short) is a hypergraph obtained by enlarging the edges arbitrarily. Gy\H{o}ri, Salia and Zamora determined the maximum number of hyperedges in a connected $r$-uniform hypergraph on $n$ vertices containing no Berge path of length $k-1$ for all $k\geq 2r+14$ and sufficiently large $n$, and asked for the minimum $k_0$ such that this extremal number holds for all $k\geq k_0$. 
In this paper, we prove that the extremal number holds for all $k\geq 2r+2$ and fails for $k\le 2r+1$, thereby completely resolving the problem posed by Gy\H{o}ri, Salia and Zamora. 
Moreover, we improve the result of F\"uredi, Kostochka and Luo, who determined the maximum number of hyperedges in a $2$-connected $n$-vertex $r$-uniform hypergraph containing no Berge cycle of length at least $k$ for all $k\geq 4r$ and sufficiently large $n$, by showing that this extremal number holds for all $k\geq 2r+2$ and fails for $k\le 2r+1$.

Our approach reduces Berge-Tur\'an problems to classical extremal graph theory problems, and applies recent work of Ai, Lei, Ning and Shi concerning the feasibility of graph parameters and the Kelmans operation. 
\end{abstract}



\section{Introduction}
An \textit{$r$-uniform hypergraph} ($r$-graph for short) $\cH=(V(\cH),E(\cH))$ consists of a vertex set $V(\cH)$ and a hyperedge set $E(\cH)$, where each hyperedge in $E(\cH)$ is an $r$-subset of $V(\cH)$. For simplicity, let $e(\cH):=|E(\cH)|$. The \textit{degree} $d_\cH(v)$ of a vertex $v$ is the number of hyperedges containing $v$ in $\cH$.

Let $\mathcal{F}$ be a family of $r$-graphs. An $r$-graph $\cH$ is called \textit{$\mathcal{F}$-free} if $\cH$ does not contain any member in $\mathcal{F}$ as a subhypergraph. The \textit{Tur\'{a}n number} ${\rm{ex}}_r(n,\mathcal{F})$ of $\mathcal{F}$ is the maximum number of hyperedges in an $\mathcal{F}$-free $r$-graph on $n$ vertices. If $\mathcal{F}=\{G\}$, then we write ${\rm{ex}}_r(n,G)$ instead of ${\rm{ex}}_r(n,\{G\})$. When $r=2$, we write $\mathrm{ex}(n,\mathcal{F})$ instead of $\mathrm{ex}_2(n,\mathcal{F})$.

Given a graph $F$, an $r$-graph $\cH$ is a \textit{Berge-$F$} if there is a bijection $\phi: E(F)\rightarrow E(\cH)$ such that $e\subseteq \phi(e)$ for each $e\in E(F)$. For a fixed graph $F$, many hypergraphs are Berge-$F$. For convenience, we refer to this collection of hypergraphs as ``Berge-$F$''. We call the vertices of $F$ the \textit{defining vertices} and we call the hyperedges $\phi(e)$ the \textit{defining hyperedges} of the Berge-$F$.
Berge \cite{A2} defined the Berge cycle, and Gy\H{o}ri, Katona and Lemons \cite{B6} defined the Berge path. Later, Gerbner and Palmer \cite{B3} generalized the established concepts of Berge cycle and Berge path to general graphs.

We first consider the case where $r=2$. Let $P_k$ denote the path on $k$ vertices, and let $\cC_{\ge k}$ denote the family of cycles with length at least $k$. 
In 1959, Erd\H{o}s and Gallai \cite{ErG} proved the following results for $\mathrm{ex}(n, P_k)$ and $\mathrm{ex}(n,\mathcal{C}_{\geq k})$.

\begin{thm}[Erd\H{o}s and Gallai \cite{ErG}]\label{thm1.1}
Fix integers $n$ and $k$ such that $n \geq k \geq 2$. Then $\operatorname{ex}(n, P_k) \leq \frac{(k-2)n}{2}$ with equality holding if and only if $G$ is the disjoint union of complete graphs on $k-1$ vertices.
\end{thm}

\begin{thm}[Erd\H{o}s and Gallai \cite{ErG}]\label{thm1.2}
Fix integers $n$ and $k$ such that $n \geq k \geq 3$. Then
$\operatorname{ex}(n, \mathcal{C}_{\geq k}) \leq \frac{(k-1)(n-1)}{2}.$
\end{thm}

Note that the extremal graph in Theorem \ref{thm1.1} is not connected.
When restricting our attention to connected graphs, Kopylov~\cite{K} determined the value of $\ex^{conn}(n,P_k)$, where $\ex^{conn}(n,P_k)$ denotes the Tur\'{a}n number for connected graphs avoiding a path of length $k$. Subsequently, Balister, Gy\H{o}ri, Lehel and Schelp \cite{BGLS} strengthened Kopylov's result by fully characterizing the extremal graphs for all $n$. 

\begin{thm}[Kopylov \cite{K}, Balister, Gy\H{o}ri, Lehel, Schelp \cite{BGLS}]
Fix integers $n \ge k \ge 5$. Then
$$
\ex^{conn}(n, P_k) = \max\left\{
\binom{k-2}{2} + (n - k + 2),\;
\binom{\left\lceil \frac{k}{2} \right\rceil}{2} + \left(\left\lfloor \frac{k-2}{2} \right\rfloor\right)\left(n - \left\lceil \frac{k}{2} \right\rceil\right)
\right\}.
$$
\end{thm}

In \cite{K}, Kopylov also obtained the following result. 

\begin{thm}[Kopylov \cite{K}]\label{thm1.4}
Let \( n \geqslant k \geqslant 5 \) and let \( t = \left\lfloor \frac{k-1}{2} \right\rfloor \). If \( G \) is a 2-connected \( n \)-vertex graph with $$
e(G) > \max\left\{\binom{k-2}{2}+2(n-k+2),\binom{k-t}{2}+t(n-k+t)\right\},$$
then \( G \) has a cycle of length at least \( k \).
\end{thm}

Now let us turn our attention to hypergraph analogues of paths and cycles. Gy\H{o}ri, Katona and Lemons \cite{B6} generalized the Erd\H{o}s-Gallai theorem to Berge paths. Specifically, they determined ${\rm{ex}}_r(n,{\rm Berge}{\text -}P_{k})$ for the cases when $k>r+2>4$ and $r\geq k-1>2$. The case when $k=r+2>4$ was settled by Davoodi, Gy\H{o}ri, Methuku and Tompkins \cite{10-10A1}.

\begin{thm}[Gy\H{o}ri, Katona, Lemons\cite{B6}, Davoodi, Gy\H{o}ri, Methuku, Tompkins \cite{10-10A1}]\label{gykl} 
\item[\textbf{(i)}] If $k\ge r+2>4$, then ${\rm{ex}}_r(n,{\rm Berge}{\text -}P_{k})\leq \frac{n}{k-1}\binom{k-1}{r}$. Furthermore, this bound is sharp whenever $k-1$ divides $n$.
\item[\textbf{(ii)}] If $r\geq k-1>2$, then ${\rm{ex}}_r(n,{\rm Berge}{\text -}P_{k})\leq \frac{n(k-2)}{r+1}$. Furthermore, this bound is sharp whenever $r+1$ divides $n$.
\end{thm}
Observe however, that these bounds are sharp only in the case that the above divisibility conditions hold. Gy\H ori, Lemons, Salia and Zamora \cite{GLSZ} showed that ${\rm{ex}}_r(n,{\rm Berge}{\text -}P_{k})=\left\lfloor \frac{n}{r+1}\right\rfloor(k-2)+\mathbf{1}_{r+1\,|\,n+1}$ if $4\le k\le r+1$, where $\mathbf{1}_{r+1\,|\,n+1}=1$ if $r+1\,|\,n+1$, and $\mathbf{1}_{r+1\,|\,n+1}=0$ otherwise. Note that if $k=3$, then ${\rm{ex}}_r(n,{\rm Berge}{\text -}P_3)=\left\lfloor \frac{n}{r}\right\rfloor$. 
Very recently, Cheng, Gerbner, Hama Karim, Miao and Zhou \cite{CGHMZ} determined the exact Tur\'{a}n number of Berge paths for the case where $k\ge r+2$. 

\begin{thm}[Cheng, Gerbner, Hama Karim, Miao and Zhou \cite{CGHMZ}]\label{CGHMZ}
    Let $k\ge r+2$ and $n=p(k-1)+q$ with $q<k-1$. Then $\ex_r(n,\textup{Berge-}P_{k})=p\binom{k-1}{r}+\binom{q}{r}$.
\end{thm}

Observe that in the above theorems, the extremal hypergraphs are not connected. Here we say that a hypergraph is \textit{connected} if we cannot partition the vertex set into two parts with no hyperedge containing vertices from both parts.
We denote by $\ex^{conn}_r(n,\cF)$ the maximum number of hyperedges in a connected $r$-graph on $n$ vertices that does not contain any copy of $F$ as a subhypergraph for all $F\in \cF$. Gy\H ori, Methuku, Salia, Tompkins and Vizer \cite{GMSTV} obtained bounds on $\ex^{conn}_r(n,\textup{Berge-}P_{k})$. F\"{u}redi, Kostochka and Luo \cite{FKL2} determined $\ex^{conn}_r(n,\textup{Berge-}P_{k})$ for all sufficiently large $n$ when $k\ge 4r+1\geq 13$. Subsequently, the threshold was improved to $k\ge 2r+14\geq 20$ by Gy\H ori, Salia and Zamora \cite{GSZ}. Gerbner et al. \cite{GNPSV} established a stability result for $\ex^{conn}_r(n,\textup{Berge-}P_{k})$.

\begin{thm}[Gy\H ori, Salia and Zamora \cite{GSZ}]\label{conn-path}
    For all integers $n,k$ and $r$, there exists $N_{k,r}$ such that if $n>N_{k,r}$ and $k\geq 2r+14\geq 20$, then
    $$\ex^{conn}_r(n,\textup{Berge-}P_{k})=\binom{\lf k/2 \rf-1}{r-1}(n-\lf k/2\rf+1)+\binom{\lf k/2\rf-1}{r}+\mathbf{1}_{2|k-1}\binom{\lfloor k/2 \rfloor-1}{r-2}.$$
\end{thm}

It is straightforward to verify that $\binom{\lf k/2 \rf-1}{r-1}(n-\lf k/2\rf+1)+\binom{\lf k/2\rf-1}{r}+\mathbf{1}_{2|k-1}\binom{\lfloor k/2 \rfloor-1}{r-2}=\binom{\lf k/2 \rf-1}{r-1}(n-\lc k/2\rc)+\binom{\lc k/2\rc}{r}$. 
Let $\cH(n,k,r)$ denote the following hypergraph. We fix a vertex set $L$ with $|L|=\lf k/2\rf -1$, and add $n-|L|$ vertices. We take as hyperedges all the $r$-sets that contain at least $r-1$ vertices in $L$. When $k$ is even, these are all the hyperedges. When $k$ is odd, we add the $r$-sets containing two fixed vertices $u_1,u_2\notin L$ and $r-2$ vertices in $L$. This is an extremal hypergraph in Theorem \ref{conn-path}.

In this work, we improve the threshold in Theorem \ref{conn-path} to $k\geq 2r+2$. 

\begin{thm}\label{thm: con Pk}
    For all integers $n,k$ and $r$, there exists $N_{r,k}$ such that if $n>N_{r,k}$ and $k\geq 2r+2\geq 8$, then
    $$\ex^{conn}_r(n,\textup{Berge-}P_k)=\binom{\lf k/2 \rf-1}{r-1}(n-\lc k/2\rc)+\binom{\lc k/2 \rc}{r}.$$
\end{thm}

Here, $\cH(n,k,r)$ achieves the above bound.
We show that the bound $k\geq 2r+2$ is sharp for this extremal construction.
Indeed, when $k\in \{2r,2r+1\}$, we have $|L|=r-1$, and thus the longest Berge path in $\cH(n,k,r)$ has $r+1$ vertices if $k=2r$ and $r+2$ vertices if $k=2r+1$. This implies that we can add at least one hyperedge to $\cH(n,k,r)$ that is contained in $V(\cH(n,k,r))\setminus L$, and the longest Berge path in the resulting hypergraph has at most $r+2<2r\leq k$ vertices if $k=2r$ and $r+3<2r+1\leq k$ vertices if $k=2r+1$. When $k=2r-1$, we have $|L|=r-2$, and thus $\cH(n,k,r)$ contains exactly one hyperedge. It is clearly not extremal. 
When $k<2r-1$, we have $|L|=r-2$ if $k=2r-2$ and $|L|<r-2$ if $k<2r-2$, implying that $\cH(n,k,r)$ has no hyperedges. Hence, it is not extremal.



\smallskip

Another closely related problem is forbidding all the Berge cycles of length at least $k$ (clearly, this is a weakening of forbidding a Berge-$P_k$). The largest number of hyperedges under this condition was determined in \cite{FKL,FKL3,EGMSTZ,GLSZ,kolu}. The extremal hypergraph is connected - this is obvious, since we could add a hyperedge connecting the two parts without creating any Berge cycle. However, the extremal hypergraph is not 2-connected.

We call a hypergraph $\cH$ \textit{$2$-connected} if it is connected and has neither cut vertex (i.e., a vertex $v\in V(\cH)$ for which there exists a partition of $V(\cH)=\{v\}\cup V_1\cup V_2$, $|V_i|\geq 1$, such that every hyperedge is contained in either $\{v\}\cup V_1$ or $\{v\}\cup V_2$), nor a cut hyperedge (i.e., a hyperedge $e\in E(\cH)$ for which there is a partition of $V(\cH)=V_1\cup V_2$, $|V_i|\geq 1$, such that every hyperedge $f\neq e$ is contained in either $V_1$ of in $V_2$). F\"{u}redi, Kostochka and Luo \cite{FKL2} gave the value of the maximum number of hyperedges in an $n$-vertex $2$-connected $r$-graph with no Berge cycle of length at least $k$.

\begin{thm}[F\"{u}redi, Kostochka and Luo \cite{FKL2}]\label{conn-cycle}
    For all integers $n$, $k\geq 4r\geq 12$, there exists $N_{r,k}$ such that if $n\geq N_{r,k}$ and $\cH$ is an $n$-vertex $2$-connected $r$-graph with no Berge cycle of length at least $k$, then 
    $$e(\cH)\leq \binom{\lf (k-1)/2 \rf }{r-1}(n-\lc (k+1)/2\rc)+\binom{\lc (k+1)/2\rc}{r}.$$
\end{thm}

The hypergraph  $\cH(n,k+1,r)$ achieves the above bound. Here we can extend the above result to all $k\geq 2r+2$. 

\begin{thm}\label{thm: 2-conn Bergee cycles}
    For all integers $n$, $k\geq 2r+2\geq 8$, there exists $N_{r,k}$ such that if $n\geq N_{r,k}$ and $\cH$ is an $n$-vertex $2$-connected $r$-graph with no Berge cycle of length at least $k$, then 
    $$e(\cH)\leq \binom{\lf (k-1)/2 \rf }{r-1}(n-\lc (k+1)/2\rc)+\binom{\lc (k+1)/2\rc}{r}.$$
\end{thm}

We remark that the bound $k\geq 2r+2$ is sharp, since $\cH(n,k+1,r)$ is not $2$-connected for $k\leq2r+1$.

This paper is organized as follows. In Section \ref{2}, we describe the connection of our problem to generalized Tur\'{a}n problems and Kelmans operation.
We give the proof of Theorem \ref{thm: con Pk} in Section \ref{3} and the proof of Theorem \ref{thm: 2-conn Bergee cycles} in Section \ref{4}.

\section{Preliminaries}\label{2}
For a graph $G$ and a subset of vertices $A\subseteq V(G)$, let $G-A$ denote the subgraph of $G$ induced by $V(G)\setminus A$.
An often-used tool in the study of the Tur\'{a}n number of Berge hypergraphs is its connection to generalized Tur\'{a}n problems. Given two graphs $H$ and $G$, we denote by $\cN(H,G)$ the number of copies of $H$ contained in $G$ as subgraphs. For graphs $H$ and $F$, let $\ex(n,H,F)$ denote the maximum value of $\cN(H,G)$, where $G$ is an $n$-vertex $F$-free graph. Such problems are simply called \textit{generalized Tur\'{a}n problems} and have attracted a lot of attention recently, see \cite{GePa} for a survey. 

The connection between Tur\'{a}n problems for Berge hypergraphs and generalized Tur\'{a}n problems has been established by Gerbner and Palmer \cite{GP2} by showing the simple bounds $\ex(n,K_r,F)\leq {\rm{ex}}_r(n,{\rm Berge}{\text -}F)\leq \ex(n,K_r,F)+\ex(n,F)$. Later, a stronger upper bound was obtained by F\"{u}redi, Kostochka, and Luo \cite{FKL} and independently by Gerbner, Methuku and Palmer \cite{B4}. 
Recently, Zhao et al. \cite{ZYWBZ} generalized their results to the graph family $\cF$. We consider an $\cF$-free graph $G$, and obtain a red-blue graph $G^{rb}$ by coloring each edge red or blue. Let $G_{\textrm{red}}$ denote the subgraph consisting of the red edges and $G_{\textrm{blue}}$ denote the subgraph consisting of the blue edges. We let $g_r(G^{rb}):=e(G_{\textrm{blue}}) + \mathcal{N}(K_r,G_{\textrm{red}})$.

\begin{lemma}[F\"{u}redi, Kostochka, Luo \cite{FKL}, Gerbner, Methuku, Palmer \cite{B4}, Zhao et al. \cite{ZYWBZ}]\label{gmp}
	Let $\mathcal{H}$ be a Berge-$\cF$-free $r$-graph. Then we can construct an $\cF$-free red-blue graph $G^{rb}$ 
    such that
	\[e(\mathcal{H}) \leq g_r(G^{rb}).\]
\end{lemma}

We need an additional property of $G^{rb}$, which follows readily from the proof of Lemma \ref{gmp} in \cite{B4} (see also Lemma 2.9 in \cite{ZYWBZ}). We present a proof that is essentially the proof of Lemma \ref{gmp} in \cite{B4}, for the sake of completeness. Suppose $\cH$ is a Berge-$\cF$-free $r$-graph. We consider an \textit{auxiliary bipartite graph} $H$, with part $A$ being the 2-sets of vertices in $V(\cH)$, part $B$ being the set of hyperedges in $\cH$, and $a\in A$ is joined to $b\in B$ with an edge if $b$ contains $a$.

\begin{proposition}\label{prop: matching}
    Let $\cH$ be a Berge-$\cF$-free $r$-graph. Then we can construct a $\cF$-free red-blue graph $G^{rb}$ 
    such that there is a matching $M$ between $E(G)$ and $E(\cH)$ in the auxiliary bipartite graph $H$ such that $M$ covers $E(G)$, and
	\[e(\mathcal{H}) \leq g_r(G^{rb}).\] 
    Moreover, the vertex set of $G$ is $V(\cH)$, and each hyperedge of $\cH$ contains either a blue edge or a red $K_r$ of $G$.
\end{proposition}

\begin{proof}
    Let us consider an arbitrary maximal matching $M_0$ in $H$. Let $A_1\subset A$ and $B_1\subset B$ be the sets of vertices not incident to $M$. Then there are no edges between $A_1$ and $B_1$. An \textit{alternating path} in $H$ is a path
that alternates between edges in $M$ and edges not in $M_0$ (beginning with an edge of $M_0$). It is well-known and easy to see that there is no alternating path from $A_1$ to $A\setminus A_1$ and from $B_1$ to $B\setminus B_1$.

Let $B_2\subset B$ be the set of vertices that we can reach from $B_1$ by an alternating path, and $A_2$ be the set of vertices matched to vertices of $B_2$, then from $A_2$, all the edges go to $B_2$. Similarly, let $A_3\subset A$ be the set of vertices that we can reach from $A_1$ by an alternating path, and $B_3$ be the set of vertices matched to vertices of $A_3$, then from $B_3$, all the edges go to $A_3$. Finally, let $A_4$ and $B_4$ denote the rest of the vertices.

Let us color the edges in $A_2$ blue and the edges in $A_3\cup A_4$ red.
Then the number of hyperedges is $e(\cH)=|B_1|+|B_2|+|B_3|+|B_4|=|B_1|+|A_2|+|B_3|+|B_4|=e(G_{\textrm{red}})+|B_1|+|B_3|+|B_4|$. The vertices in $B_1\cup B_3\cup B_4$ are only adjacent to vertices in $A_3\cup A_4$ in $H$. This means that they correspond to blue cliques in $G$, thus $|B_1|+|B_3|+|B_4|\le \mathcal{N}(K_r,G_{\textrm{blue}})$.
\end{proof}



We now introduce the \textit{Kelmans operation}. Given a graph $G$ and two vertices $u,v$, we obtain $G[u\rightarrow v]$ in the following way. For each vertex $x$ that is adjacent to $u$ but not $v$, we delete $ux$ and add $vx$.

A \textit{graph parameter} is a function that maps each graph to a real number. 
We say that a graph parameter $P$ is \textit{feasible} if for any $u,v$ we have $P(G)\le P(G[u\rightarrow v])$ and $P(G)<P(G+e)$ for any $e\not\in E(G)$ but $V(e)\cap V(G)\neq \emptyset$. We say that $P$ is \textit{weakly feasible} if $P(G)<P(G+e)$ is replaced by $P(G)\le P(G+e)$.

The main results of Ai, Lei, Ning and Shi \cite{alns} are the following. Let $W(n,k,s)=K_s\vee [(n-k+s)K_1\cup K_{k-2s}]$, and let $X=V(K_s)$, $Y=V(K_{k-2s})$ and $Z=V((n-k+s)K_1)$. 

\begin{thm}[Ai, Lei, Ning and Shi \cite{alns}]\label{path}
    Let $n\ge k\ge 4$ and let $t=\lfloor k/2\rfloor-1$. Let $G$ be a connected $n$-vertex $P_k$-free graph. If $P$ is weakly feasible, then $P(G)\le \max\{P(W(n,k-1,s): 1\le s\le t\}$. Moreover, if $P$ is feasible, then each connected $n$-vertex $P_k$-free graph $G$ with the maximum $P(G)$ is equal to $W(n,k-1,s)$ for some $1\le s\le t$.
\end{thm}
\begin{thm}[Ai, Lei, Ning and Shi \cite{alns}]\label{thm: ai-2-connect cycles}
    Let $n\ge k\ge 5$ and let $t=\lfloor (k-1)/2\rfloor$. Let $G$ be a $2$-connected $n$-vertex $\cC_{\ge k}$-free graph. If $P$ is weakly feasible, then $P(G)\le \max\{P(W(n,k,s): 1\le s\le t\}$. Moreover, if $P$ is feasible, then each $2$-connected $n$-vertex $\cC_{\ge k}$-free graph $G$ with the maximum $P(G)$ is equal to $W(n,k,s)$ for some $2\le s\le t$.
\end{thm}
In this paper, we build a connection between Kelmans' operation and Berge-Tur\'an problems.
Let $P(G)$ be a graph parameter. Given a red blue graph $G^{rb}$, we let $P^{rb}(G^{rb})=P(G_{\textrm{red}})+|E(G_{\textrm{blue}})|$. We define a colored version of the Kelmans operation. Given a red-blue graph $G^{rb}$ and two vertices $u,v$ of $G$, we let $G^{rb}[u\rightarrow_{rb} v]$ denote the following graph. For each vertex $x\neq u,v$, if $x$ is joined to $u$ but not $v$, then replace the edge $ux$ by the edge $vx$ of the same color. If $x$ is joined to $u$ with a red edge and to $v$ with a blue edge, then we exchange the colors of the edges $ux,vx$, i.e., $ux$ becomes blue, and $vx$ becomes red.

In \cite{alns}, the authors proved that if $G$ is $P_k$-free or $\cC_{\ge k}$-free, then $G[u\rightarrow v]$ is also $P_k$-free or $\cC_{\ge k}$-free, respectively.
Since we execute the ordinary Kelmans operation on the underlying graph, if $G$ is $P_k$-free or $\cC_{\ge k}$-free, then $G^{rb}[u\rightarrow_{rb} v]$ is also $P_k$-free or $\cC_{\ge k}$-free, respectively. 
Observe that on $G$ and on the red graph $G_{\textrm{red}}$, we executed the ordinary Kelmans operation, and the number of blue edges does not change. This implies the following. 


\begin{prop}
    Let $P(G)$ be a graph parameter and $u,v\in V(G)$. Then $P^{rb}(G^{rb}[u\rightarrow_{rb} v])=P(G_{\textrm{red}}[u\rightarrow v])+|E(G_{\textrm{blue}})|$.
\end{prop}

We say that $P^{rb}$ is \textit{feasible} if for any $u,v$ we have $P^{rb}(G^{rb})\le P^{rb}(G^{rb}[u\rightarrow_{rb} v])$. Note that we do not assume $P^{rb}(G^{rb})<P^{rb}(G^{rb}+e)$, since adding a blue edge increases $P^{rb}$ anyway.
Clearly, we have the following.
\begin{cor}
    If a graph parameter $P$ is weakly feasible, then $P^{rb}$ is feasible.
\end{cor}

Given a graph $G$, we 
denote by $P^*(G)$ the largest value of $P^{rb}(G^{rb})$ where $G^{rb}$ is a red-blue coloring of $G$.
Then we have the following connection between $P$ and $P^*$.
\begin{prop}\label{feas}
    If $P$ is weakly feasible, then $P^*$ is feasible.
\end{prop}

\begin{proof}
    Let $G$ be a graph and $G^{rb}$ be a red-blue coloring of $G$ with the largest value of $P^{rb}(G^{rb})$, i.e., $P^{rb}(G^{rb})=P^*(G)$. Let $u,v\in V(G)$ and $G'=G^{rb}[u\rightarrow_{rb} v]$. Recall that $G'_{\textrm{red}}=G_{\textrm{red}}[u\rightarrow v]$ and $|E(G'_{\textrm{blue}})|=|E(G_{\textrm{blue}})|$. Therefore, $P^*(G)=P(G_{\textrm{red}})+|E(G_{\textrm{blue}})|\leq P(G'_{\textrm{red}})+|E(G'_{\textrm{blue}})|=P^{rb}(G')$. Since $G'$ is a red-blue coloring of $G[u\rightarrow v]$, we have that $P^{rb}(G')\le P^*(G[u\rightarrow v])$, thus $P^*(G)\le P^*(G[u\rightarrow v])$.

    Let $G''$ be a red-blue coloring of $G+e$ obtained by adding a blue $e$ to $G^{rb}$. Then $P^*(G+e)\ge P^{rb}(G'')=P^{rb}(G^{rb})+1=P^*(G)+1$, proving the second condition of feasibility, thus completing the proof.
\end{proof}

Notice that the number of cliques in a graph is a weakly feasible parameter, which is proved in \cite{alns}.
Thus, for a given graph $G$, the maximum sum of the number of red cliques and the number of blue edges among all red-blue colorings is a feasible parameter.








Our goal is to apply Theorems \ref{path} and \ref{thm: ai-2-connect cycles} together with the above proposition, to give upper bounds on $g_r(G^{rb})$ for the graph in Lemma \ref{gmp}. There is only one problem with this approach: that the graph in Lemma \ref{gmp} is not necessarily connected (resp. 2-connected) even if the original hypergraph is connected (resp. 2-connected). The rest of this paper deals with overcoming this complication.
 
For a graph $G$ and a set $S\subseteq V(G)$, let $G[S]$ denote the subgraph of $G$ induced by $S$. For a vertex $v\in V(G)$, let $N_G(v)$ denote the nrighbourhood of $v$ in $G$. We omit the subscript $G$ when it is clear from the context. 

In this paper, we focus on the extremal structure of connected graphs (resp. $2$-connected graphs) with no long paths (resp. no long cycles) and achieve the maximum value of $g_r(G^{rb})$.
Theorem \ref{path} and Theorem \ref{thm: ai-2-connect cycles} show that the extremal structures are in the families of $\{W(n,k-1,s)\}_{1\leq s\leq t}$ and $\{W(n,k,s)\}_{2\leq s\leq t}$, where $t\in \left\{\lf k/2\rf-1,\lf(k-1)/2\rf \right\}$.
Thus, we determine the extremal coloring that maximizes $g_r(G^{rb})$ among all red-blue colorings of $W(n,k-1,s)$ (resp. $W(n,k,s)$) for $1\le s\le t=\lf k/2\rf-1$ (resp. $2\le s\le t=\lf(k-1)/2\rf$).


\begin{lemma}\label{lemmaa} 
Let $k\ge 2r+2\geq 8$, there is a constant $N_{r,k}\geq 4k$ such that the following holds.

\textbf{(i)}. For all $n\geq N_{r,k}$, if $G^{rb}$ is a red-blue coloring of $W(n,k-1,s)$ for some $1\le s\le t=\lf k/2\rf-1$, then the maximum value of $g_r(G^{rb})$ is achieved by the monored $W(n,k-1,t)$. Moreover, when $k=2r+2$ or $k=2r+3$, the maximum value of $g_r(G^{rb})$ is also attained by the monoblue $W(n,k-1,t)$. 

\textbf{(ii)}. For all $n\geq N_{r,k}$, if $G^{rb}$ is a red-blue coloring of $W(n,k,s)$ for some $2\le s\le t=\lf(k-1)/2\rf$, then the maximum value of $g_r(G^{rb})$ is achieved by the monored $W(n,k,t)$. Moreover, when $k=2r+2$, the maximum value of $g_r(G^{rb})$ is also attained by the monoblue $W(n,k,t)$. 
\end{lemma}

\begin{proof}
For any vertex $v$, we have the following claim to bound the number of red $r$-cliques plus the number of blue edges containing $v$.
\begin{clm}\label{clm: rb with one vertex} 
    Let $t\in \left\{\lf k/2\rf-1,\lf(k-1)/2\rf \right\}$.
    For a vertex with degree $t$ in $G^{rb}$, the number of red $r$-cliques and blue edges containing this vertex is at most $\binom{t}{r-1}$. For a vertex with degree less than $t$ in $G^{rb}$, the number of red $r$-cliques and blue edges containing this vertex is at most $\binom{t}{r-1}-1$.
\end{clm}
\begin{proof}[Proof of Claim]
For a vertex $v$ with degree $t$ that is incident to $i$ blue edges, the number of blue edges plus red $r$-cliques containing $z$ is at most $i+\binom{s-i}{r-1}$. By the convexity of the binomial coefficient, this is the largest when $i$ is 0 or $s$, i.e., at most $\max\{s, \binom{s}{r-1}\}$. Since $s\le t$, this is at most $\max\{t, \binom{t}{r-1}\}$. 
When $k\geq 2r+4$ or $k=2r+3$ and $r=\lf(k-1)/2\rf$, we have $t\ge r+1$, and thus that is at most $\binom{t}{r-1}$, with equality only if $i=0,s=t$, each edge incident to $v$ is red and $G[N(v)]$ is monored.
When $k=2r+2$ or $k=2r+3$ and $r=\lf k/2\rf-1$, we have $t=r$ and thus that is at most $\max\{r, \binom{r}{r-1}\}=r$, with equality only if $i=0,s=t$, each edge incident to $v$ is red and all the edges in $G[N(v)]$ is red, or if $i=t$ and each edge incident to $v$ are blue.

 
For a vertex with degree at most $t-1$, the number of red $r$-cliques and blue edges containing this vertex is at most $i+\binom{t-1-i}{r-1}$, which is the largest when $i$ is 0 or $t-1$, i.e., at most $\max\{t-1, \binom{t-1}{r-1}\}$. Since $t\ge r$, this is at most $\binom{t}{r-1}-1$.
\end{proof}

Let us return to the proof of \textbf{(i)}. 
Recall that $W(n,k-1,s)=K_s\vee [(n-k+s+1)K_1\cup K_{k-2s-1}]$, and let $X=V(K_s)$, $Y=V(K_{k-2s-1})$ and $Z=V((n-k+s+1)K_1)$.
Then we apply Claim \ref{clm: rb with one vertex} to each vertex of $Z$.
Then, each vertex has at most $\binom{t}{r-1}$ red $r$-cliques and blue edges containing it (since $s\leq t$), and the total contribution of vertices in $Z$ is at most $(n-k+s+1)\binom{t}{r-1}$.
For every vertex of $Z$, when $k\geq 2r+4$, 
if we recolor the incident edges to red, then $g_r(G^{rb})$ does not decrease. Similarly, if a vertex of $Y$ is joined to some vertices of $X$ by blue edge, then we can recolor those edges to red without decreasing $g_r(G^{rb})$. After that, only the edges inside $Y$ may be blue, but every such edge would form a red $K_r$ with $r-2$ vertices of $X$, thus we may recolor them to red. This way we obtain a monored $W(n,k-1,t)$ without decreasing $g_r(G^{rb})$.

When $k=2r+2$ or $k=2r+3$, 
if all the edges contained in $X$ are red, then we recolor the edges incident to $Z$ to red, then $g_r(G^{rb})$ does not decrease. Then we recolor all the other edges to red and similarly obtain a monored $W(n,k-1,t)$ without decreasing $g_r(G^{rb})$. If some of the edges contained in $X$ are blue, then we recolor all the edges incident to $Z$ to blue, and recolor all the other edges to blue. This way we obtain a monoblue $W(n,k-1,t)$ without decreasing $g_r(G^{rb})$.
It is easy to check in both cases that the value of $g_r(G^{rb})$ is the same for monored or monoblue $W(n,k-1,t)$, thus we are done. 

By a similar argument to that of \textbf{(i)}, we may prove \textbf{(ii)}. This completes the proof. 
\end{proof}

Combining the above lemma with Theorems \ref{path} and \ref{thm: ai-2-connect cycles}, we have the following corollary.

\begin{cor}\label{cor: extremal structure}
    Let $k\ge 2r+2\geq 8$. For any connected $P_k$-free graph (resp. $2$-connected $\cC_{\ge k}$-free graph) $G$ on at least $N_{r,k}$ vertices with a red-blue coloring $G^{rb}$, the maximum value of $g_r(G^{rb})$ is achieved by the monored $W(n,k-1,t)$ (resp. $W(n,k,t)$), where $t=\lfloor k/2\rfloor-1$ (resp. $t=\lfloor (k-1)/2\rfloor$) and $N_{r,k}$ is the constant given in Lemma \ref{lemmaa}.
\end{cor}

Recall that a graph is connected if and only if there is a path between any pair of vertices. 
The analogous statement for hypergraphs and Berge paths is well-known, and we include a proof for the sake of completeness. 

\begin{lemma}\label{conn}
    A hypergraph $\cH$ is connected if and only if there is a Berge path between any pair of vertices. 
\end{lemma}

\begin{proof} 
    Clearly, if $\cH$ is disconnected, then no Berge path connects vertices from different components. 
    Now assume that $\cH$ is connected, and let $u,v$ be two vertices.
    Let $V_1$ be the set of vertices $w$ such that there exists a Berge path connecting $v$ and $w$, and let $V_2=V(\cH)\setminus V_1$. 
    If $u\in V_1$, then we are done. Now suppose $u\in V_2$. 
    Since $\cH$ is connected, there is a hyperedge $e$ that intersects both $V_1$ and $V_2$. Suppose $w_1\in V_1\cap e$ and $w_2\in V_2\cap e$. Then $e$ is not a defining hyperedge in the Berge path from $v$ to some $w\in V_1\backslash\{v\}$. Otherwise, $|e\cap V_1|\geq2$ and we can find a Berge path from $v$ to $w_2$, a contradiction. Thus, we can extend the Berge path from $v$ to $w_1$ by adding the hyperedge $e$ and the defining vertex $w_2$, yielding a Berge path from $v$ to $w_2$, a contradiction. 
\end{proof}
A \emph{block} in a graph \(G\) is a maximal connected subgraph \(G'\) with no cut vertices (of \(G'\)) \cite{FKL2}. A block is called a \emph{leaf block} if it contains at most one cut vertex of $G$. 
In the graph case, we have the following result.
\begin{lemma}\label{lem: connect_find_path}
    Let $G$ be a connected graph with at least $t+2$ vertices, and $\delta(G)\geq t$, then for every vertex $v$, there is a path with at least $t+2$ vertices starting from $v$, unless $v$ is a cut vertex of a block which is a clique with $t+1$ vertices. 
\end{lemma}
\begin{proof}
    Let $P$ be the length path starting from $v$ with vertices $v=v_1,v_2,\dots,v_s$. 
    Then, we have $N(v_s)\subseteq \{v_1,\dots,v_{s-1}\}$, otherwise we can find a longer path.
    Since $\delta(G)\geq t$, we have $s-1\ge t$, and $s\ge t+1$.
    But if $s=t+2$, then we have find the path we want, thus $s=t+1$, and $N(v_s)=\{v_1,\dots,v_s\}$.

    As a result, every vertex $v_i$ in $\{v_2,\dots,v_s\}$ lies in a path with $t+1$ vertices starting from $v=v_1$ and ends at $v_i$.
    Then $N(v_i)=\{v_1,\dots,v_s\}\setminus \{v_i\}$ for every $i=2,\dots,s$.
    It implies $\{v_1,\dots,v_s\}$ forms a clique.
    Then, $v$ is a cut vertex of the clique formed by $\{v_1,\dots,v_s\}$.
\end{proof}

We now recall an analogous equivalent characterization of 2-connectedness for graphs. 
\begin{thm}[Whitney \cite{Wh}]\label{Whitney}
An undirected graph $G$ of order $n \ge 3$ is $2$-connected if and only if for any two distinct vertices $u, v \in V(G)$, there exist at least two internally vertex-disjoint $u-v$ paths in $G$. Here, internally vertex-disjoint means that the two paths share no common vertices except the endpoints $u$ and $v$.
\end{thm}

For hypergraphs, we have the following analogous version for 2-connected hypergraphs. 

\begin{lemma}\label{lem: 2 berge path conn 2 vertices}
    Let $n\geq r\geq 3$ and $\cH$ be a $2$-connected $r$-graph on $n$ vertices. Then for any $u,v\in V(\cH)$, there exist two disjoint Berge paths (sharing no defining vertices except the endpoints $u$,$v$ and no defining hyperedges) between $u$ and $v$. 
\end{lemma}
\begin{proof}
    It is equivalent to proving that there exists a Berge cycle containing $u$ and $v$ as defining vertices.
    We prove it by induction on the distance of $u$ and $v$, i.e., the length of the minimum Berge path connecting $u$ and $v$.
    When the distance of $u$ and $v$ is $1$, it implies there exists a hyperedge $h$ containing $u$ and $v$.
    Delete hyperedge $h$, since $\cH$ is $2$-connected, the resulting hypergraph $\cH'$ is connected. By Lemma \ref{conn} there exists a Berge path $P$ in $\cH'$ connecting $u$ and $v$, then $h$ and $P$ form a Berge cycle containing $u$ and $v$ as defining vertices.

    Suppose that the distance between $u$ and $v$ is $k > 1$.
    Let $P_{uv}=u=w_0,e_1,w_1,\dots,w_{k-1},e_k,v$ be the shortest Berge path connecting $u$ and $v$.
    Then, the distance of $u,w_{k-1}$ is $k-1$.
    By the induction hypothesis, there exists a Berge cycle $C_{u,w_{k-1}}$ containing $u$ and $w_{k-1}$ as defining vertices.
    If $v$ is a defining vertex of $C_{u,w_{k-1}}$, then we are done. Thus, we may assume $v$ is not a defining vertex of $C_{u,w_{k-1}}$.
    
    Then, since $\cH$ is $2$-connected, we can similarly prove that there exists a shortest Berge path $P'$ avoiding the vertex $w_{k-1}$, with defining hyperedges $g_1,\dots, g_\ell$, and connecting $v$ and some defining vertex $z$ of $C_{u,w_{k-1}}$, where $z\neq w_{k-1}$.
    We assume $v\in g_1,$ and $z\in g_\ell$.
    Then for $j<\ell$, $g_j$ is not a defining hyperedge of $C_{u,w_{k-1}}$, and does not contain any defining vertices of $C_{u,w_{k-1}}$, otherwise we find a shorter path, a contradiction.

    If $g_\ell$ is a defining hyperedge of $C_{u,w_{k-1}}$, then $g_\ell$ contains two defining vertices $x_1,x_2$. Without loss of generality, the  sub-Berge path $P''$ of $C_{u,w_{k-1}}$ from $x_1$ to $u$ avoids the vertices $\{x_2,w_{k-1}\}$. Then we pick $z:=x_1$. From $v$, we can go to $x_1$ through $P'$, and then to $u$ through $P''$, or go to $w_{k-1}$ using $e_k$ and then to $u$ in the other direction inside $C_{u,w_{k-1}}$.
    
    If $g_\ell$ is not a defining hyperedge of $C_{u,w_{k-1}}$, then we can simply go from $z$ to $u$ using the sub-Berge path of $C_{u,w_{k-1}}$ from $z$ to $u$ avoids the vertices $w_{k-1}$, and the rest of the argument is the same.
        This completes the proof.
\end{proof}

The above lemma can be easily extended to the setting of two disjoint vertex sets as follows, which will be used in the proof of Theorem \ref{thm: 2-conn Bergee cycles}. 
Let $\cH$ be an $r$-graph, and let $S_1,S_2\subseteq V(\cH)$ be disjoint vertex sets with $|S_i|\geq 2$. 

We say that there exist two disjoint Berge paths between $S_1$ and $S_2$ if there exist two disjoint Berge paths from $u_1$ to $v_1$ and from $u_2$ to $v_2$ with $u_1,u_2\in S_1$ and $v_1,v_2\in S_2$, that share no defining vertices and no defining hyperedges. 

\begin{lemma}\label{lem: 2 berge path conn}
    Let $n\geq r\geq 3$ and $\cH$ be a $2$-connected $r$-graph on $n$ vertices. Suppose that $S_1,S_2\subseteq V(\cH)$ are two disjoint vertex sets with $|S_i|\geq 2$. Then 
    there exist two disjoint Berge paths between $S_1$ and $S_2$.
\end{lemma}
\begin{proof}
    Add two new vertices $s_1$ and $s_2$. 
    For each vertex $v \in S_1$, add a hyperedge containing $v$, $s_1$, and \(r-2\) new distinct vertices. Similarly, for each vertex \(v \in S_2\), add a hyperedge containing \(v\), \(s_2\), and \(r-2\) distinct new vertices. 
    The resulting hypergraph contains $1+(r-2)|S_i|$ new vertices associated with each set $S_i$. Now add all hyperedges contained in these $(r-2)|S_1|$ new vertices and $(r-2)|S_2|$ new vertices, respectively. 
    It is easy to check that the resulting hypergraph is still $2$-connected by the definition of $2$-connected. 
    By Lemma \ref{lem: 2 berge path conn 2 vertices}, there exist two disjoint Berge paths between $s_1$ and $s_2$ in the resulting hypergraph. This implies that there exist two disjoint Berge paths between $S_1$ and $S_2$ in $\cH$, completing the proof. 
\end{proof}

\section{Connected Tur\'an number of Berge paths}\label{3}
In this section, we provide the proof of Theorem \ref{thm: con Pk}, which we restate here for convenience. 

\begin{thmbis}{thm: con Pk}
    For all integers $n,k$ and $r$, there exists $N_{r,k}$ such that if $n>N_{r,k}$ and $k\geq 2r+2\geq 8$, then
    $$\ex^{conn}_r(n,\textup{Berge-}P_k)=\binom{\lf k/2 \rf-1}{r-1}(n-\lc k/2\rc)+\binom{\lc k/2 \rc}{r}.$$
\end{thmbis}
\begin{proof}
    Let $\cH$ be a connected $n$-vertex Berge-$P_k$-free $r$-graph. By Lemma \ref{gmp}, we can construct a $P_k$-free red-blue graph $G^{rb}$ such that $e(\cH)\le g_r(G^{rb})$. Here we pick among such graphs $G^{rb}$ a red-blue graph such that $G$ has the fewest connected components. Let $G$ be the underlying graph of $G^{rb}$. 
    We will apply Proposition \ref{feas} with $P(G)=N(K_r,G)$, thus $P^{rb}(G^{rb})=g_r(G^{rb})$. 


It was shown in \cite{alns} that $P$ is weakly feasible, therefore, $P^*$ is feasible. By Theorem \ref{path}, if $G$ is connected, then we have $P^*(G)\le \max\{P^*(W(n,k-1,s)): 1\le s\le t\}$. Therefore, 
$$e(\cH)\le g_r(G^{rb})=P^{rb}(G^{rb})\le P^*(G)\le \max\{P^*(W(n,k-1,s)): 1\le s\le t\}.$$ Then Lemma \ref{lemmaa}\,\textbf{(i)} completes the proof.
Therefore, we can assume that $G$ is disconnected.
Now we classify the components.
\begin{itemize}
    \item A component of $G$ is \textit{nice} if each vertex in it has degree at least $\lf k/2 \rf -1$, and contains no block that is a clique with size $\lf k/2 \rf$.
    \item Consider a component $U$ that is not nice.
    We remove the vertices of degree less than $\lf k/2\rf -1$ from $U$, and the non-cut vertices in blocks that are cliques with size $\lf k/2\rf $.
    Then we repeat this in the remaining part of $U$, and continue this until we are left with a subset $U'\subset U$ where every degree is at least $\lf k/2 \rf -1$, and without blocks that are cliques with size $\lf k/2 \rf$. We say that $U$ is \textit{strong} if $U'$ is not empty and $U'\neq U$.
    \item The other components are \textit{bad}.
\end{itemize}
\begin{clm}\label{clm: path in nice and strong}
    In a nice component, every vertex is the starting vertex of a path of length at least $\lf k/2 \rf$. In a strong component $U$, 
    every vertex in $U'$ is the starting vertex of a path of length at least $\lf k/2 \rf$, and every vertex in $U\setminus U'$ is the starting vertex of a path of length at least $\lf k/2 \rf+1$.
\end{clm}
    \begin{proof}[Proof of Claim]
     First, we show that in a nice component, every vertex is the starting vertex of a path of length at least $\lf k/2 \rf$ by Lemma \ref{lem: connect_find_path}, because there is no block which is a clique with size $\lf k/2 \rf$.

     By the same argument, in a strong component, each vertex in $U'$ is the starting vertex of a path of length at least $\lf k/2 \rf $ inside $U'$. For each vertex $u\in U\setminus U'$, there is a shortest path to $U'$, and from the end of that path, there is a path of length at least $\lf k/2 \rf$ inside $U'$. Therefore, there is a path of length at least $\lf k/2 \rf+1$ starting at $u$.
     \end{proof}
     Then, we have the following claim.
   \begin{clm}
       There is at most one component in $G$ that is nice or strong.
   \end{clm}  

   \begin{proof}[Proof of Claim]    
   Assume that $U$ and $W$ are nice or strong components. Since $\cH$ is connected, by Lemma \ref{conn} there is a shortest Berge path connecting them with hyperedges $h_1,\dots, h_\ell$, such that $h_1$ contains a vertex of $U$ and $h_\ell$ contains a vertex of $W$. Because this is the shortest such Berge path, only $h_1$ contains one or more vertices of $U$ and only $h_\ell$ contains one or more vertices of $W$. We take the longest path starting at a vertex of $U\cap h_1$ inside $U$, and consider the corresponding hyperedges of $\cH$. We also take the longest path starting at a vertex of $W\cap h_\ell$ inside $W$ and the corresponding hyperedges.
   According to Claim \ref{clm: path in nice and strong}, these two paths have length at least $\lf k/2 \rf$.
   This extends $h_1,\dots,h_\ell$ to a Berge path of length at least $k-1$, unless a hyperedge is used multiple times. This can only happen if it is an $h_i$ and also $M(a)=h_i$ for some edge $a$ inside $U$ or $W$, where $M$ is the matching in Proposition \ref{prop: matching}. Therefore, we either have that $i=1$ and $a$ is an edge inside $U$, or $i=\ell$ and $a$ is an edge inside $W$, or both. We assume without loss of generality that $h_1=M(a)$ for some $a=uu'$ used by the path inside $U$.


   We claim that $uu'$ cuts $U$ into two components. Indeed, 
otherwise we could change $uu'$ to a blue edge $uv$ with $v\in h_1\setminus U$ to obtain another red-blue graph satisfying the desired properties with fewer components, contradicting our choice of $G$. Then $uu'$ also cuts $U'$ into two components. Let $U_1$ denote the component containing $u$ after removing the edge $uu'$. We apply our greedy algorithm inside $U_1$, starting at $u$, which gives a path of length at least $\lf k/2 \rf$, starting at $u$. Obviously, this path does not contain $uu'$.

Similarly, if $h_\ell=M(a')$ for some edge $a'=ww'$ inside $W$, then we can find a path of length at least $\lf k/2 \rf$, starting at $w$, that does not contain $ww'$. Then the hyperedges corresponding to these two paths together with $h_1,\dots,h_\ell$ form a Berge path of length at least $2\lf k/2 \rf\geq k-1$, a contradiction. 
   \end{proof}

   Let us return to the proof of the theorem. For components that are neither nice nor strong, we remove the vertices and blocks that are $\lf k/2 \rf$-clique one by one. Recall that each time we removed a vertex $v$ (of degree less than $\lf k/2 \rf-1$), it has degree at most $\lf k/2 \rf-2$ at that point. According to Claim \ref{clm: rb with one vertex}, we removed at most $\binom{\lfloor (k-1)/2\rfloor}{r-1}-1$ red $r$-cliques and blue edges.
   And we removed the non-cut vertices of the blocks, which are cliques of size $\lf k/2 \rf$, the number of red $r$-cliques and blue edges is at most $\binom{ \lf k/2 \rf}{r}\leq (\lf k/2\rf -1)\left (\binom{\lf k/2 \rf-1}{r-1}-1\right)$. It implies that, on average, when deleting a vertex, the number of red $r$-cliques and blue edges decreases at most $\binom{\lf k/2 \rf-1}{r-1}-1$.

  In the remaining single component, we have already established the desired bound. Let $n'$ be the number of vertices in that component.
  If $n'\geq N'_{r,k}$, where $N'_{r,k}$ is the constant in Lemma \ref{lemmaa}, then by Corollary \ref{cor: extremal structure}, the total number of red $r$-cliques and blue edges is at most
  \begin{align*}
   & \binom{\lceil k/2\rceil}{r}+(n'-\lceil k/2\rceil)\binom{\lfloor k/2\rfloor -1}{r-1}+(n-n')\left(\binom{\lf k/2 \rf -1}{r-1}-1\right)\\
   <& \binom{\lceil k/2\rceil}{r}+(n-\lceil k/2\rceil)\binom{\lfloor k/2\rfloor -1}{r-1},
  \end{align*}
  and we are done.
  If $n'<N'_{r,k}$, then we have at most $\binom{N'_{r,k}}{r}+\binom{N'_{r,k}}{2}$ red $r$-cliques and blue edges in $G'$.
  Thus, the total number of red $r$-cliques and blue edges is at most 
  \begin{align*}
    &\binom{N'_{r,k}}{r}+\binom{N'_{r,k}}{2}+(n-n')\left(\binom{\lfloor k/2\rfloor-1}{r-1}-1\right)\\
   <& \binom{\lceil k/2\rceil}{r}+(n-\lceil k/2\rceil)\binom{\lfloor k/2\rfloor -1}{r-1}.
  \end{align*}
The inequality holds for sufficiently large $ n$, and we are done.
\end{proof}

\section{$2$-connected Tur\'an number of long Berge cycles}\label{4}

We will use the following strengthening of the Erd\H os-Gallai Theorem.
\begin{lemma}[Li and Ning \cite{li2021strengthening}]\label{lem: x,y path}
    Let $G$ be a $2$-connected graph on $n$ vertices, and $x,y\in V(G)$. If there are at least $\frac{n-1}{2}$ vertices in $V(G)\setminus \{x,y\}$ of degree at least $s$, then $G$ contains a $(x,y)$-path with at least $s+1$ vertices.
\end{lemma}

Furthermore, we also need the following stability result about the Tur\'an number of $C_{\ge k}$.

\begin{thm}[F\"uredi, Kostochka and Verstra\"ete \cite{furedi2016stability}]\label{thm: more general stability}
    Let $n\geq k\geq 5$ and $t=\lf \frac{k-1}{2}\rf$. If $G$ is an $n$-vertex $2$-connected graph with no cycle of length at least $k$, then
    $e(G)\leq e(W(n,k,t-1)$, unless
\begin{itemize}
    \item $k=2t+1, k\neq 7$, and $G\subseteq W(n,k,t)$ or
    \item $k=2t+2$ or $k=7$, and $G-A$ is a star forest for some $A\subseteq V(G)$ of size at most $t$.
\end{itemize}
\end{thm}
For the case when $k=7$, there is a more specific result for its structure. First, we need to define the following families of graphs.
Each $G\in \cG_2(n,k)$ is defined by a partition $V(G)=A\cup B\cup J$, $|A|=t$ and a pair $a_1\in A, b_1\in B$ such that $G[A]=K_t$, $G[B]$ is the empty graph, $G(A,B)$ is a complete bipartite graph and for every $c\in J$ one has $N(c)=\{a_1,b_1\}$.
Each $G\in \cG_3(n,k)$ is defined by a partition $V(G)=A\cup B\cup J$, $|A|=t$ such that $G[A]=K_t$, $G(A,B)$ is a complete bipartite graph, and
\begin{itemize}
    \item $G[J]$ has more than one component
    \item all components of $G[J]$ are stars with at least two vertices each
    \item  there is a $2$-element subset $A'$ of $A$ such that $N(J)\cap (A\cup B)=A'$
    \item for every component $S$ if $G[J]$ with at least $3$ vertices, all leaves of $S$ are adjacent to the same vertex $a(S)$ in $A'$.
\end{itemize}

Note that this is a general definition in \cite{furedi2016stability}, but we will only consider $\mathcal{G}_2(n,6)$ and $\mathcal{G}_3(n,6)$. In particular, $t=2$ thus $A=A'$ in our case.

\begin{thm}[F\"uredi, Kostochka and Verstra\"ete \cite{furedi2016stability}]\label{thm: for small k=7}
Let $n \geq 7$, and $G$ be an $n$-vertex $2$-connected graph with no cycle of length at least $7$. Then either $e(G) \leq e(W(n,7,2))$ or $G$ is a subgraph of a graph in $\mathcal{G}(n,7)$, where
    $$\mathcal{G}(n,7) := \{W(n,7,3)\} \cup \{W(n,6,2)\} \cup \mathcal{G}_2(n,6) \cup \mathcal{G}_3(n,6).$$
\end{thm}
\noindent
Then, we have the following lemma.

\begin{lemma}\label{lem: when k=8}
    Suppose $G$ is a $2$-connected, $n$-vertex graph with no cycle of length at least $k$, where $k$ is odd, $n\geq 4k$ and $e(G)\geq ((k+1)/2-4/3)n$.
    Then for every two vertices $u,v\in V(G)$, there is a path with at least $(k+1)/2+1$ vertices connecting $u,v$.
\end{lemma}
\begin{proof}
First, we deal with the case $k\geq 9$.
If $G$ contains no cycles of length at least $k$, then since $e(G)\geq ((k+1)/2-\frac{4}{3})n$, according to Theorem \ref{thm: more general stability}, $G$ is the subgraph of $W(n,k-1, (k-1)/2)$ when $k\geq 9$ is odd. 

Let $X,Y,Z$ be the set of vertices as the definition of $W(n,k-1,(k-1)/2-1)$, and $|X|=(k-1)/2-1$, $|Y|=0$ and $|Z|=n-(k-1)/2+1$. 
By the lower bound of $e(G)$ and $n\geq 4k$, one can easily check that the number of common neighbours of $X$ is at least $k$.
Moreover, we have $\delta(G)\geq 2$, otherwise $G$ is not $2$-connected.
Then one can easily check that for every two vertices $u,v$, there is a path with at least $(k+1)/2+1$ vertices connecting them.

When $k=7$, then we have $e(G)\geq (4-4/3)n=8n/3$. Note that $e(W(n,7,2))=2(n-2)+\binom{2}{2}+\binom{3}{2}=2n$. Then by Theorem \ref{thm: for small k=7}, $G$ is a subgraph of a graph in $\mathcal{G}(n,7)$. With simple calculation, we obtain that when $n\geq 4k=28$,
    \begin{eqnarray*}
    &{}& e(W(n,6,2))=2(n-2)+2=2n-2<8n/3. \\
    &{}& e(\cG_2(n,6))=2(n-2-|J|)+\binom{2}{2}+2|J|=2n-3<8n/3. 
    \end{eqnarray*}

For $G\in \cG_3(n,6)$, let $J_1$ be the union of star components $S$ in $G[J]$ with at least $3$ vertices, and $J_2$ be the star components $S$ in $G[J]$ with $2$ vertices.
Then, the number of edges incident to $J_1$ is at most $2|J_1|$, because we can stepwise delete the leaf vertices, and at each step we delete at most two edges. 
After that, when deleting the center of a star of $J_1$, we delete at most two edges.
Each component in $J_2$ is incident with at most $5$ edges, thus the number of edges incident to $J_2$ is at most $5|J_2|/2$.
Each vertex in $B$ has degree two since they are adjacent only to the two vertices in $A$. Therefore, we have
$$e(G)\leq 2(n-|A|-|B|-|J_1|-|J_2|)+2|B|+2|J_1|+5|J_2|/2+1\leq 5n/2+1<8n/3.$$   

    Thus, $G$ can only be a subgraph of $W(n,7,3)$, then, similar to the case when $k\geq 9$, it can be easily checked that for every two vertices $u,v$, there is a path with at least $(k+1)/2+1$ vertices connecting them. 
\end{proof}

Recall that
a \emph{block} in a graph \(G\) is a maximal connected subgraph \(G'\) with no cut vertices (of \(G'\)) \cite{FKL2}. A block is called a \emph{leaf block} if it contains at most one cut vertex of $G$. 
It is well-known that if a non-empty graph is not $2$-connected, then it has at least two leaf blocks.

Let us continue with the poof of Theorem \ref{thm: 2-conn Bergee cycles}, which we restate here for convenience. 

\begin{thmbis}{thm: 2-conn Bergee cycles}
    For all integers $n$, $k\geq 2r+2\geq 8$, there exists $N_{r,k}$ such that if $n\geq N_{r,k}$ and $\cH$ is an $n$-vertex $2$-connected $r$-graph with no Berge cycle of length at least $k$, then 
    $$e(\cH)\leq \binom{\lf (k-1)/2 \rf }{r-1}(n-\lc (k+1)/2\rc)+\binom{\lc (k+1)/2\rc}{r}.$$
\end{thmbis}

\begin{proof}
Let $\cH$ be a $2$-connected $n$-vertex Berge-$\cC_{\ge k}$-free $r$-graph. We construct a $\cC_{\ge k}$-free red-blue graph $G^{rb}$ such that $e(\cH)\leq g_r(G^{rb})$ using Proposition \ref{prop: matching}.
We pick among such graphs $G^{rb}$ a red-blue graph such that $G$ has the fewest number of components, and for graphs with the same number of components, we pick the one with the fewest blocks.
We still apply Proposition \ref{feas} with $P(G)=N(K_r,G)$, thus $P^{rb}(G^{rb})=g_r(G^{rb})$.
Recall that $P^*$ is feasible. By Theorem \ref{thm: ai-2-connect cycles}, if $G$ is $2$-connected, then we can complete the proof according to Lemma \ref{lemmaa}\,\textbf{(ii)}. Therefore, we can assume that $G$ is not $2$-connected. 
We may assume the number of red $r$-cliques plus blue edges in $G^{rb}$ is more than $\binom{\lfloor (k-1)/2\rfloor}{r-1}(n-\lceil (k+1)/2\rceil)+\binom{\lceil (k+1)/2\rceil}{r}$, otherwise we are done.

Since $G$ is not $2$-connected, there exist at least two leaf blocks.

We partition the leaf blocks into three types: nice, strong, and bad. For bad blocks, we further define a subclass called troublesome. The definitions are given below.
\begin{itemize} 
    \item For a leaf block $B$ of $G$, we say it is \textit{nice} if each non-cut vertex in it has degree at least $\lf k/2\rf$.
    \item For a leaf block of $G$ that is not nice, we remove the non-cut vertices with degree less than $\lf k/2\rf$ one by one, until we are left with a subgraph $B'$. We call $B$ \textit{strong} if 
    $B'$ contains a non-cut vertex of $G$. 
    \item Otherwise, we call the leaf block $B$ \textit{bad}.
    \begin{itemize}
        \item For a bad leaf block $B$, suppose it has at least $N_{r,k}\geq 4k$ vertices, where $N_{r,k}$ is the constant in Lemma \ref{lemmaa}, and $B$ has at least $(\lf k/2\rf -\frac{4}{3})(|V(B)|-1)$ edges. Then we call $B$ \textit{troublesome}.
    \end{itemize}
\end{itemize}

\begin{clm}\label{clm: nice block find path}
    In a nice or strong leaf block, for every two vertices $x,y$, there exists a path connecting $x$ and $y$ with at least $\lf k/2 \rf+1$ vertices.
\end{clm}
\begin{proof}[Proof of Claim]
Suppose $B$ is a strong leaf block with $u$ as the unique cut vertex, and $B'$ is the subgraph of $B$ where every non-cut vertex has degree at least $\lf k/2\rf$.
Then note that $B'$ is nonempty but may not be $2$-connected, nor even connected.
But when $B'$ is not $2$-connected, there always exists a leaf block in $B'$ that does not contain $u$, denoted by $B_1'$. And if $B'$ is $2$-connected, then we rename $B'$ as $B_1'$.

If $x,y$ are in $V(B_1')$, then we will apply Lemma \ref{lem: x,y path}. 
When deleting the vertices in $B\setminus B'$, the size of $B_1'$ is at least $\lf k/2\rf+1\geq 5$, and the number of vertices in $V(B_1')\setminus \{x,y\}$  with degree at least $\lf k/2\rf$ is at least $|V(B_1')|-3$, because the vertices in $V(B_1')\setminus \{x,y\}$ in addition to the unique cut vertex of $B_1'$ in $B'$ (which is $u$ when $B'$ is $2$-connected) have degree at least $\lf k/2\rf$. Since $|V(B_1')|-3\geq \frac{|V(B_1')|-1}{2}$ (resp. $|V(B')|-3\geq \frac{|V(B')|-1}{2}$), according to Lemma \ref{lem: x,y path}, there exists a path with at least $\lf k/2 \rf+1$ vertices connecting $x$ and $y$.

If $x\in V(B_1')$ but $y\notin  V(B_1')$, then since $B$ is $2$-connected, there exists a shortest path from $y$ to $B_1'$ avoiding $x$. Let $y'$, denote the intersection of $B_1'$ and this path, then $x,y'$ are in $B_1'$. As in the above case, there exists a path with at least $\lf k/2 \rf +1$ vertices connecting $x,y'$ in $B_1'$, thus we find a path with the desired length connecting $x,y$. 

If $x,y$ are both not in $V(B_1')$, then since $B$ is $2$-connected, Theorem \ref{Whitney} implies there exist two disjoint paths from $x,y$ to $B_1'$. Suppose that the two paths intersect $B_1'$ at $x'$ and $y'$, respectively, then by the above analysis, there exists a path with at least $\lf k/2 \rf +1$ vertices connecting $x',y'$, and it can be extended to a path with the desired length connecting $x$ and $y$.
Thus, for every two vertices $x,y\in B$, there exists a path with at least $\lf k/2 \rf+1$ vertices connecting them.

If $B$ is a nice leaf block, then $B=B'$ and we are done by Lemma \ref{lem: x,y path} with a similar analysis as above. 
\end{proof}

Our strategy is to delete the bad leaf blocks one by one, until there are no bad leaf blocks. This way we obtain a series of graphs $G=G_0,G_1,,G_2,\dots,G_M$ with $G_i\subseteq G_{i-1}$ for $i\geq 1$. We will analyse the structure of the remaining graph $G_M$.
Note that during the deletion process, deleting a whole leaf block may create new leaf blocks. We classify the new leaf blocks as described above.

During the deletion process, we have the following claim.
\begin{lemma}\label{lem: at most one nice or strong leaf block}
For $i=0,\dots,M$,
    there is at most one leaf block in $G_i$ that is nice or strong.
\end{lemma}
\begin{proof}
Assume that $B_1$ and $B_2$ are nice or strong leaf blocks. If $B_i$ is strong, let us delete the non-cut vertices with degree less than $\lf k/2\rf $ one by one in $B_i$, and let
$B_i'$ denote one of the remaining leaf blocks afterwards. We rename $B_i'$ to $B_i$.

Since $\cH$ is $2$-connected, if $B_1$ and $B_2$ are disjoint, according to Lemma \ref{lem: 2 berge path conn} there exist two shortest disjoint Berge paths (with different defining hyperedges and defining vertices) connecting $B_1,B_2$, denoted by $h_1,\dots,h_\ell$ and $g_1,\dots,g_m$.
Then there exist $c_1\in h_1\cap B_1$, $c_2\in h_\ell\cap B_2$ and $d_1\in g_1\cap B_1$, $d_2\in g_m\cap B_2$.
If $B_1$ and $B_2$ share a common cut vertex $v$, then we can assume $c_1=c_2=v$, and there is a Berge path $h_1,\dots,h_\ell$ connecting $B_1$ and $B_2$ that avoids $v$.

By Claim \ref{clm: nice block find path}, there exists a path $P^i$ of length at least $\lf k/2 \rf$ contained in $B_i$ connecting $c_i,d_i$.
Since we picked shortest Berge paths connecting $B_1$ and $B_2$, only $h_1,g_1$ contain vertices in $B_1$, and only $h_\ell,g_\ell$ contain vertices in $B_2$.
Then, we claim that there exists an edge $e_1$ in the path of length at least $\lf k/2 \rf$ contained in $B_1$, such that $M(e_1)\in\{h_1,g_1\}$ or there exists $e_2$ in that path of length at least $\lf k/2 \rf$ contained in $B_2$ such that $M(e_2)\in\{h_\ell,g_\ell\}$, where $M$ is the matching in Proposition \ref{prop: matching} (recall that every distinct edge $e$ in $G$ corresponds to a unique hyperedge $M(e)$ in $M$ that contains $e$). Indeed, otherwise, if $B_1$ and $B_2$ are disjoint, then the hyperedges $\{M(e)\}_{e\in P^1}$, together with $h_1,\dots,h_\ell$, $g_1,\dots,g_m$ and $\{M(e)\}_{e\in P^2}$ form a Berge cycle of length at least $2(\lf k/2 \rf)+2\geq k+1$, a contradiction. If $B_1$ and $B_2$ share a common cut vertex $v$, then the hyperedges $\{M(e)\}_{e\in P^1}$ together with $h_1,\dots,h_\ell$ and  $\{M(e)\}_{e\in P^2}$ form a Berge cycle of length at least $2(\lf k/2 \rf)+1\geq k$, a contradiction. 
Finally, we will prove that 
there exists no $e_1\in B_1$ such that $M(e_1)\in\{h_1,g_1\}$, nor $e_2\in B_2$ such that $M(e_2)\in\{h_\ell,g_\ell\}$, which yields a contradiction. 

\begin{clm}\label{clm: no such edge}
     There exists no $e_1\in B_1$ such that $M(e_1)\in\{h_1,g_1\}$, nor $e_2\in B_2$ such that $M(e_2)\in\{h_\ell,g_\ell\}$.
\end{clm}
\begin{proof}[Proof of Claim]
Otherwise, without loss of generality, we assume that $e_1=u_1u_2\subseteq B_1$ is such that $M(e_1)\in \{h_1,g_1\}$ exists, and we may assume $M(e_1)=h_1$.
Then we claim that the subgraph $B_1-e_1$ is not $2$-connected. Otherwise, we could reduce the number of blocks by changing the edges $u_1u_2$ to a blue edge $u_1w$, where $w\in h_1\setminus B_1$, and thus contradicting the choice of $G$. 

Let us focus on $B_1-e_1$.
Assume first that $B_1-e_1$ contains a cut edge, say $a_1a_2$, dividing $B_1$ into two $2$-connected components $B_1^1,B_1^2$. We can assume that $u_1\in B_1^1$, $u_2\in B_1^2$.
Let $B_1^{i'}$ be the sub-block of $B_1^i$ obtained from $B_1^i$ by deleting the vertices in $B_1\setminus B_1'$ for $i=1,2$.

Let $v_1$ be the cut vertex of $G$ in $B_1$ (if it exists). Then the vertices in $B_1^{i'}\setminus \{v_1,a_1,a_2,u_1,u_2\}$ each have degree at least $\lf k/2 \rf$ in $B_1^{1'}$ for $i=1,2$. Thus, with a similar proof as in Claim \ref{clm: nice block find path}, the number of vertices in $B_1^i\setminus\{a_i,u_i\}$ is at least $\frac{|V(B_1^i)|-1}{2}$ with degree  at least $\lf k/2 \rf$ for $i=1,2$.
  Then, there are two paths with at least $\lf k/2 \rf+1$ vertices connecting $a_1,u_1$ in $B_1^1$ and $a_2,u_2$ in $B_1^2$ respectively.

This way we found a Berge cycle with length at least $2(\lf k/2\rf+1)\geq k$ (see Figure \ref{fig: cycle1}), a contradiction. 

If there exists a cut vertex $a$ but no cut edge in $B_1-e_1$, then $a$ divides $B_1-e_1$ into $B_1^1,B_1^2$.
In this case, similarly, we can find two paths with at least $\lf k/2\rf +1$ vertices connecting $a,u_1$ and $a,u_2$ respectively.

Then we find a cycle of length at least $2(\lf k/2 \rf)+1\geq k$ as above, a contradiction (see Figure \ref{fig: cycle2}.) 
Thus, there is no such $e_1$, and symmetrically, there is no such $e_2$, which completes the proof of the claim. 
\end{proof}
This completes the proof of Lemma \ref{lem: at most one nice or strong leaf block}.
\end{proof}

\begin{figure}[t]
    \centering
    \begin{subfigure}[b]{0.45\textwidth}
        \centering
\includegraphics[width=\linewidth]{ 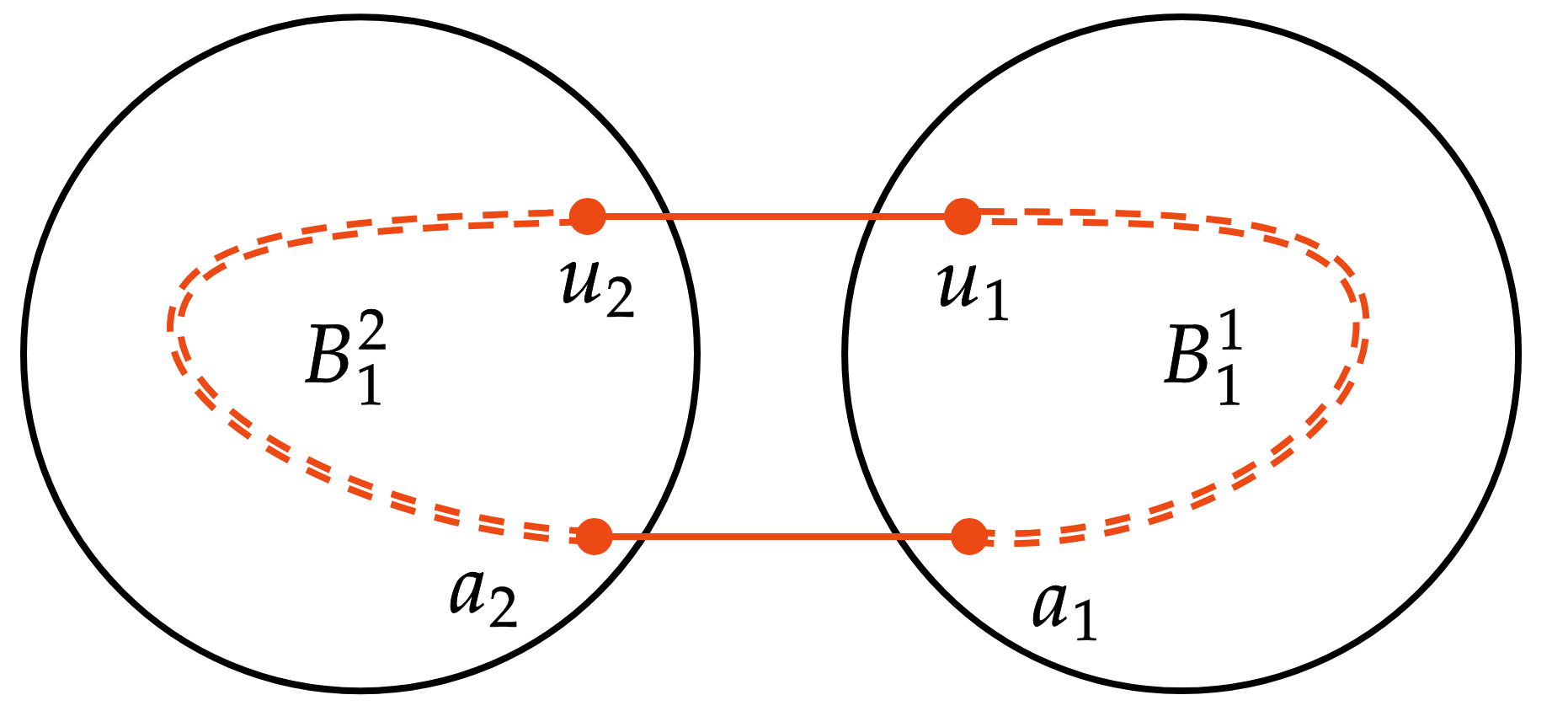}
        \caption{The case when $B_1-u_1u_2$ has a cut edge $a_1a_2$}\label{fig: cycle1}
    \end{subfigure}
    \hspace{0.01\textwidth}
    \begin{subfigure}[b]{0.42\textwidth}
\includegraphics[width=\textwidth]{ 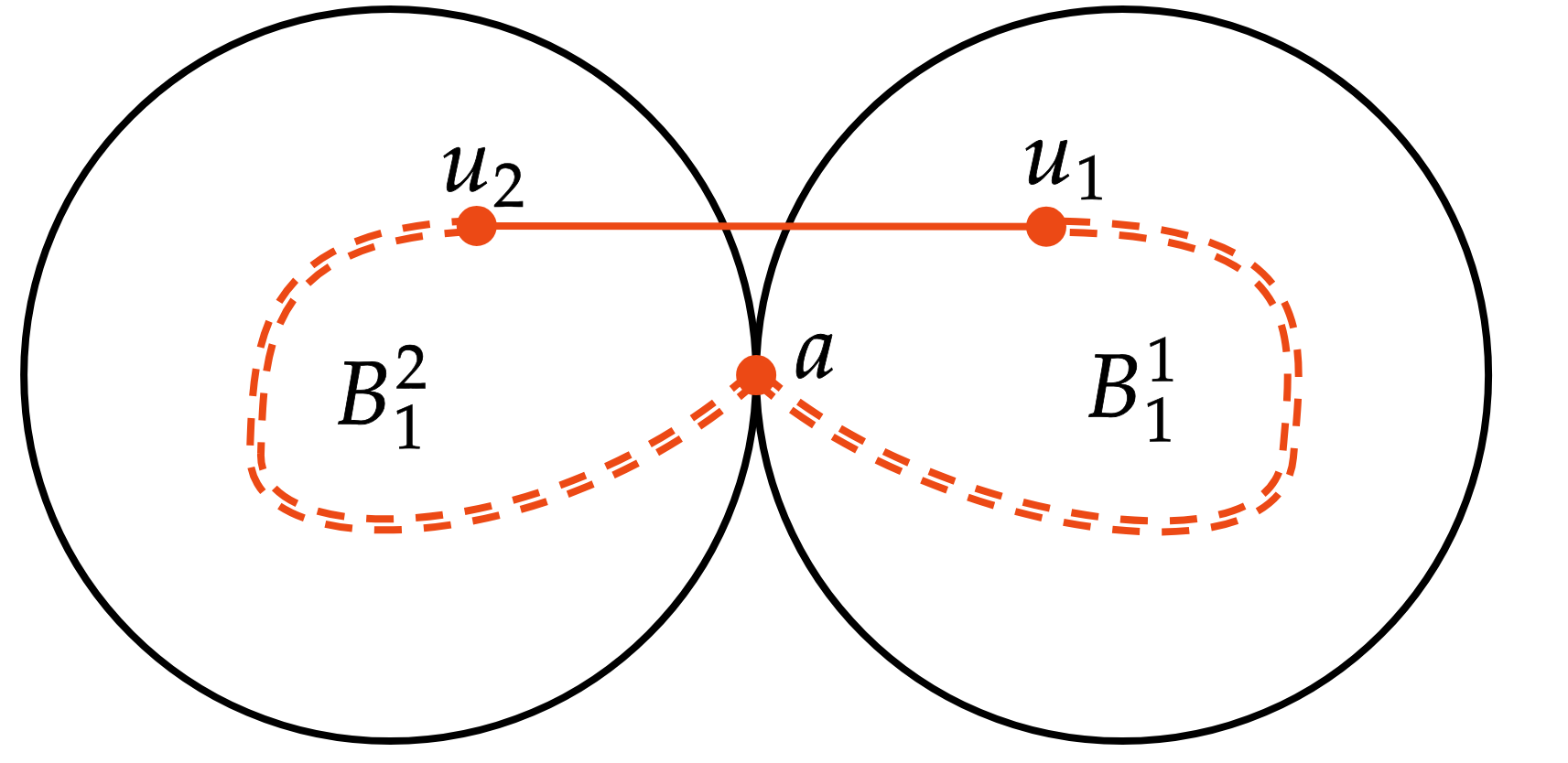}
        \caption{The case when $B_1-u_1u_2$ has a cut vertex $a$}\label{fig: cycle2}
    \end{subfigure}
\caption{}
\end{figure}

We now distinguish cases based on the parity of $k$. 

\smallskip
\noindent\textbf{If $k$ is odd.}
\smallskip
\smallskip

Clearly, $\lf k/2\rf-1< \lf(k-1)/2\rf$. 
Let $G'$ be the graph obtained by removing all the vertices with degree less than $\lf k/2\rf$ in $G$.
Then, according to Lemma \ref{lem: at most one nice or strong leaf block}, there is at most one nice or strong leaf block in $G'$.
Let $n'$ denote the number of vertices in $G'$. Similar to Claim \ref{clm: rb with one vertex}, for a vertex $v$ with degree at most $\lf k/2 \rf-1$ in $G$, the number of red $r$-cliques and blue edges containing $v$ is at most $i+\binom{\lf k/2 \rf-1-i}{r-1}$, where $i$ is the number of blue edges incident to $v$ at that point. By the convexity of the binomial coefficient, this is the largest when $i$ is 0 or $\lf k/2 \rf-1$. One can easily check that by our assumption on $k$, we have that we removed at most $\binom{\lf k/2 \rf-1}{r-1}$ red $r$-cliques and blue edges.
Similarly, for a vertex $v$ with degree less than $\lf k/2 \rf-1$ in $G$, the number of red $r$-cliques and blue edges containing $v$ is at most $\binom{\lf k/2 \rf-1}{r-1}-1$.

If $n'\geq N_{r,k}$, then by Corollary \ref{cor: extremal structure}, the number of red $r$-cliques and blue edges in $G'$ is at most $\binom{\lceil k/2\rceil}{r}+(n'-\lceil k/2\rceil)\binom{\lfloor k/2\rfloor -1}{r-1}$.
 Since in the deleting process, we delete vertices with degree at most $\lf k/2 \rf-1$, 
 the total number of red $r$-cliques and blue edges in $G$ is at most $$\binom{\lceil k/2\rceil}{r}+(n-\lceil k/2\rceil)\binom{\lfloor k/2\rfloor -1}{r-1},$$ and we are done. 
 
Now we suppose $n'<N_{r,k}$. Then the number of red $r$-cliques and blue edges in $G'$ is at most $\binom{N_{r,k}}{r}+\binom{N_{r,k}}{2}$, and the total number of red $r$-cliques and blue edges in $G$ is at most
\begin{align*}
    &\binom{N_{r,k}}{r}+\binom{N_{r,k}}{2}+(n-n')\binom{\lf k/2\rf-1}{r-1}\\
    <&\binom{\lc (k+1)/2\rc}{r-1}+(n-\lc (k+1)/2\rc)\binom{\lf (k-1)/2\rf}{r-1}
\end{align*}
for sufficiently large $n$, and we are done. 
Thus, it remains to consider the case where $k$ is even.

\smallskip
\noindent\textbf{If $k$ is even.}
\smallskip
\smallskip

It follows that $\lf k/2\rf-1=\lf (k-1)/2\rf$. We will show that there is at most one nice or strong or troublesome leaf block in $G$.

\begin{clm}\label{clm: no troublesome}
    Suppose $B_1$ is a troublesome leaf block. Then there is no other leaf block that is nice or strong or troublesome during the deleting process.
\end{clm}
\begin{proof}[Proof of Claim]
Suppose $B_1$ is a troublesome leaf block, and $B_2$ is another leaf block which is nice, strong or troublesome. Since we have $e(B_1)\geq (k/2-4/3)(|V(B_1)|-1)$, by Theorem \ref{thm1.4}, there is a cycle of length at least $k-2$ in $B_1$, denoted by $C^1$. 

Let us consider two hyperedges $h,g$ such that $h$ contains a vertex $c\in V(C^1)$ and $g$ contains $d\in V(C_1)$. Then we have the following possibilities. 

\begin{itemize}
    \item \textbf{Situation 1.} There is no edge $e_1$ in $C^1$ with $M(e_1)\in \{h,g\}$. Then there is a path $P^{1,0}$ connecting two different vertices in $h$ and $g$,
 and with at least $k/2$ vertices.
 \item \textbf{Situation 2.}
There is exactly one edge $e_1$ in $C^1$ with $M(e_1)\in \{h,g\}$. Then there is a path $P^{1,1}$ connecting two different vertices in $h$ and $g$, avoiding $e_1$ and with at least $k/2$ vertices.
\item \textbf{Situation 3.} 
There are two edges $e_1,e_1'$ in $C^1$ such that $\{M(e_1),M(e_1')\}=\{h,g\}$. Then there is a path $P^{1,2}$ connecting two different vertices in $h$ and $g$, avoiding $e_1,e_1'$
 and with at least $k/2-1$ vertices.
 \end{itemize}

 Indeed, in Situation 1, one of the two subpaths of $C^1$ connecting $c$ and $d$ satisfies these properties. In Situation 2, let $e_1=x_1x_2$, then it is easy to check that for all possible $d\in C^1$, there exists a path in $C^1$ with at least $k/2$ vertices connecting $d,x_1$ or $d,x_2$ and avoiding the edge $e_1$. In Situation 3, suppose $e_1=x_1x_2$ and $e_1'=x_1'x_2'$.
We may assume that $x_1,x_2,x_2',x_1'$ are in clockwise order in $C^1$.
Then, it is easy to check that there is a path with at least $k/2-1$ vertices, avoiding $e_1,e_1'$ and connecting $x_2,x_2'$ or $x_1',x_1$ .

\smallskip
\noindent\textbf{Case 1.} $B_2$ is nice or strong.
\smallskip

Then, there are two cases on whether $C^1$ and $B_2$ share a common cut vertex.

\smallskip
\noindent \textbf{Subcase 1.1.} $C^1$ and $B_2$ do not share a common cut vertex.
\smallskip

Then, according to Lemma \ref{lem: 2 berge path conn}, there are two disjoint Berge shortest paths connecting $C^1$ and $B_2$, denoted by $h_1,\dots,h_\ell$ and $g_1,\dots,g_m$.
Moreover, only $h_1$ and $g_1$ intersect with $V(C^1)$, and only $h_\ell,g_m$ intersect with $V(B_2)$.
Then there exist $c_1\in h_1\cap C^1$, $c_2\in h_\ell\cap B_2$ and $d_1\in g_1\cap C^1$, $d_2\in g_m\cap B_2$ that are the defining vertices in the two Berge paths.
Then, according to Lemma \ref{lem: x,y path}, there is a path (denoted $P^2$) with at least $k/2+1$ vertices connecting $c_2,d_2$ in $B_2$.
        According to Claim \ref{clm: no such edge}, we know that there is no edge $e_2\in B_2$ such that $M(e_2)\in \{h_\ell,g_m\}$. 

Let $h=h_1$ and $g=g_1$.
In each of the three situations above, we find a path connecting a vertex of $h_1$ to a vertex of $g_1$ inside $B_1$, with at least $k/2-1$ vertices. Then the hyperedges $M(e)$ for the edges $e$ in this path or in $P^2$, together with $h_1,\dots,h_\ell$, and $g_1,\dots,g_m$ form a Berge cycle with at least $k$ vertices, a contradiction.

\begin{figure}[t]
    \centering
    \begin{subfigure}[b]{0.42\textwidth}
        \centering
        \includegraphics[width=\linewidth]{ 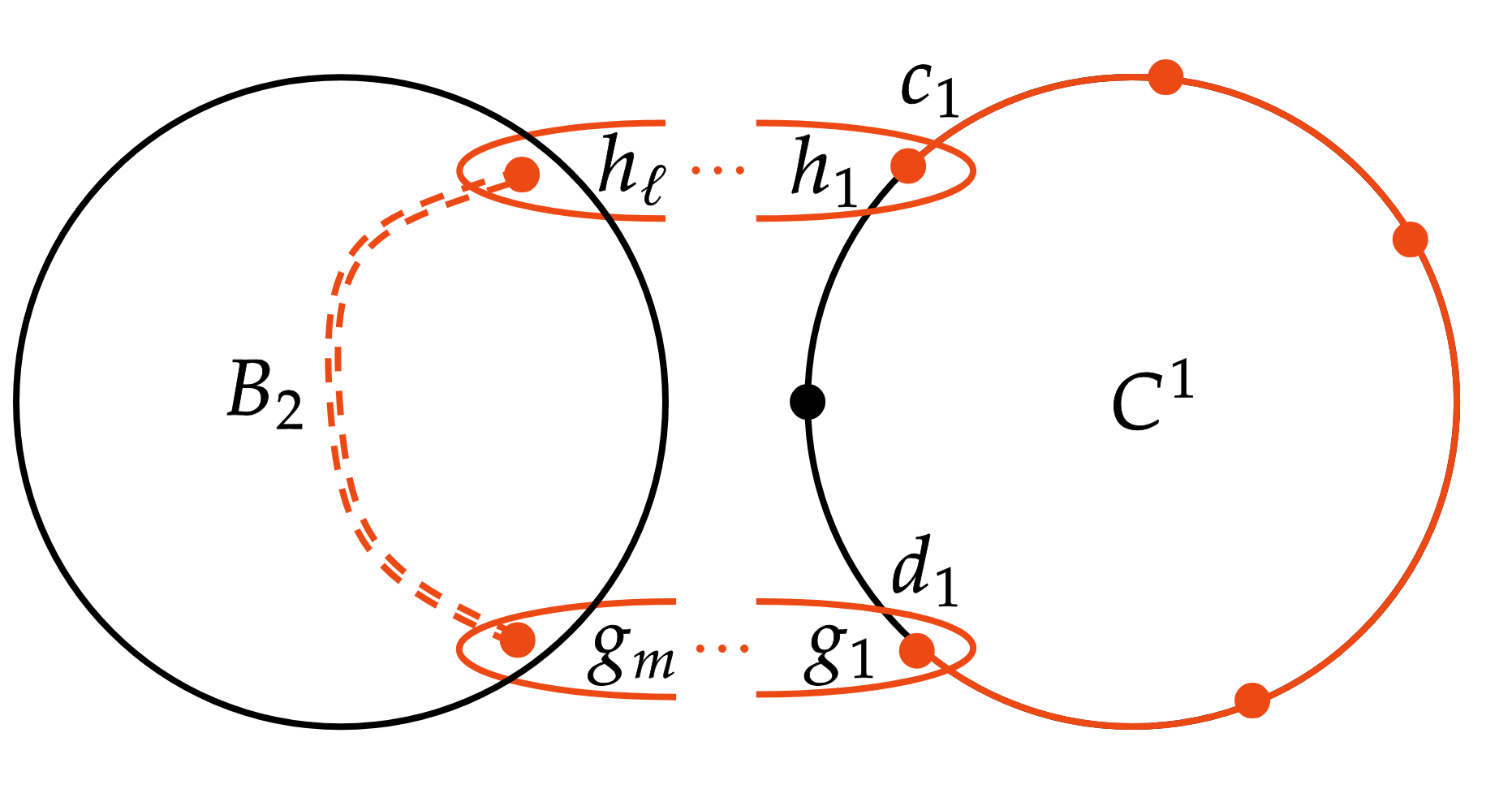}
        \caption{
        The cycle obtained in Situation 1}\label{fig: P10}
    \end{subfigure}
    \hspace{0.01\textwidth}
    \begin{subfigure}[b]{0.42\textwidth}
        \includegraphics[width=\textwidth]{ 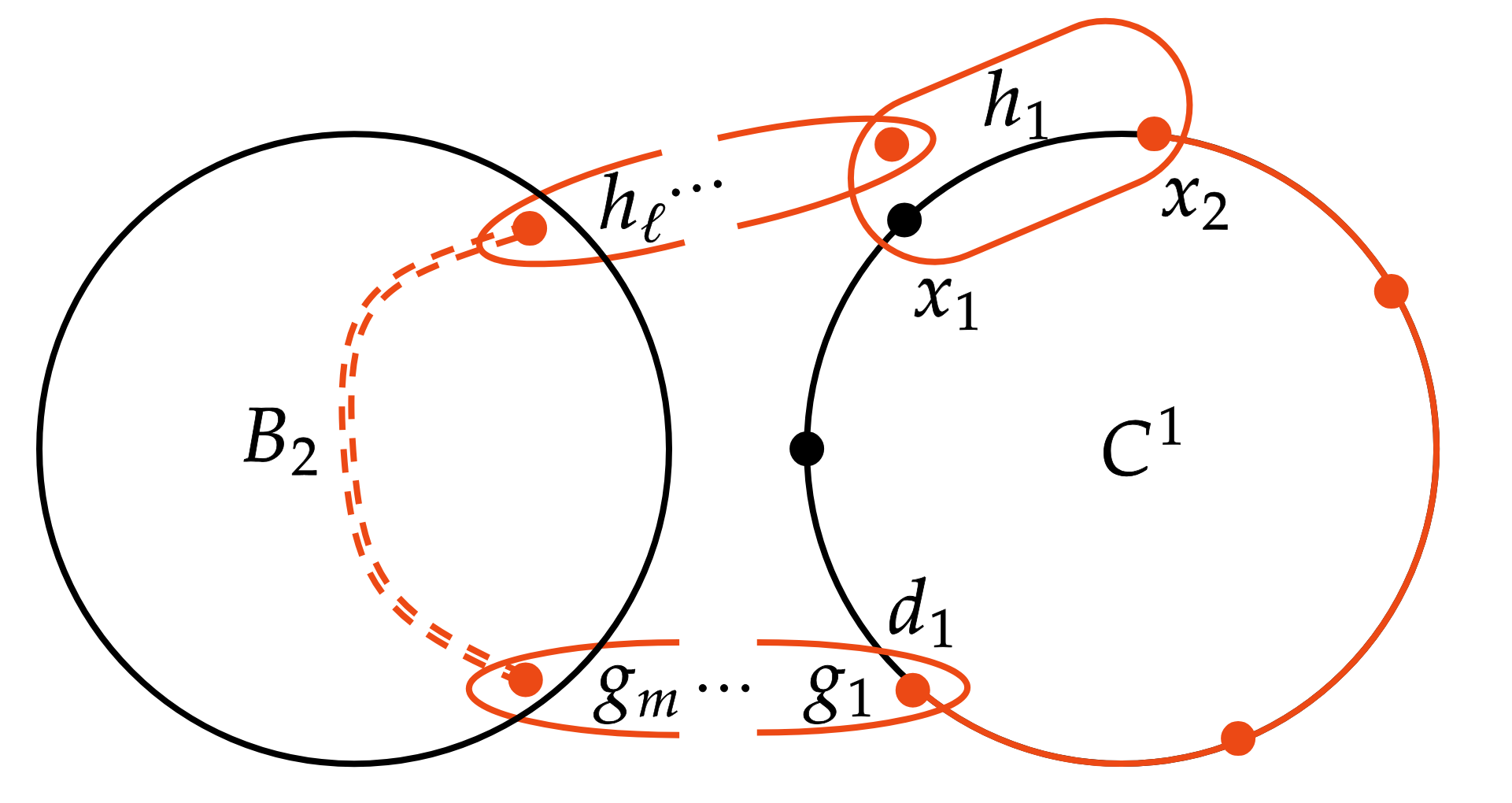}
        \caption{
        The cycle obtained in Situation 2}\label{fig: P11}
    \end{subfigure}
    \hspace{0.01\textwidth}
    \begin{subfigure}[b]{0.42\textwidth}
        \includegraphics[width=\textwidth]{ 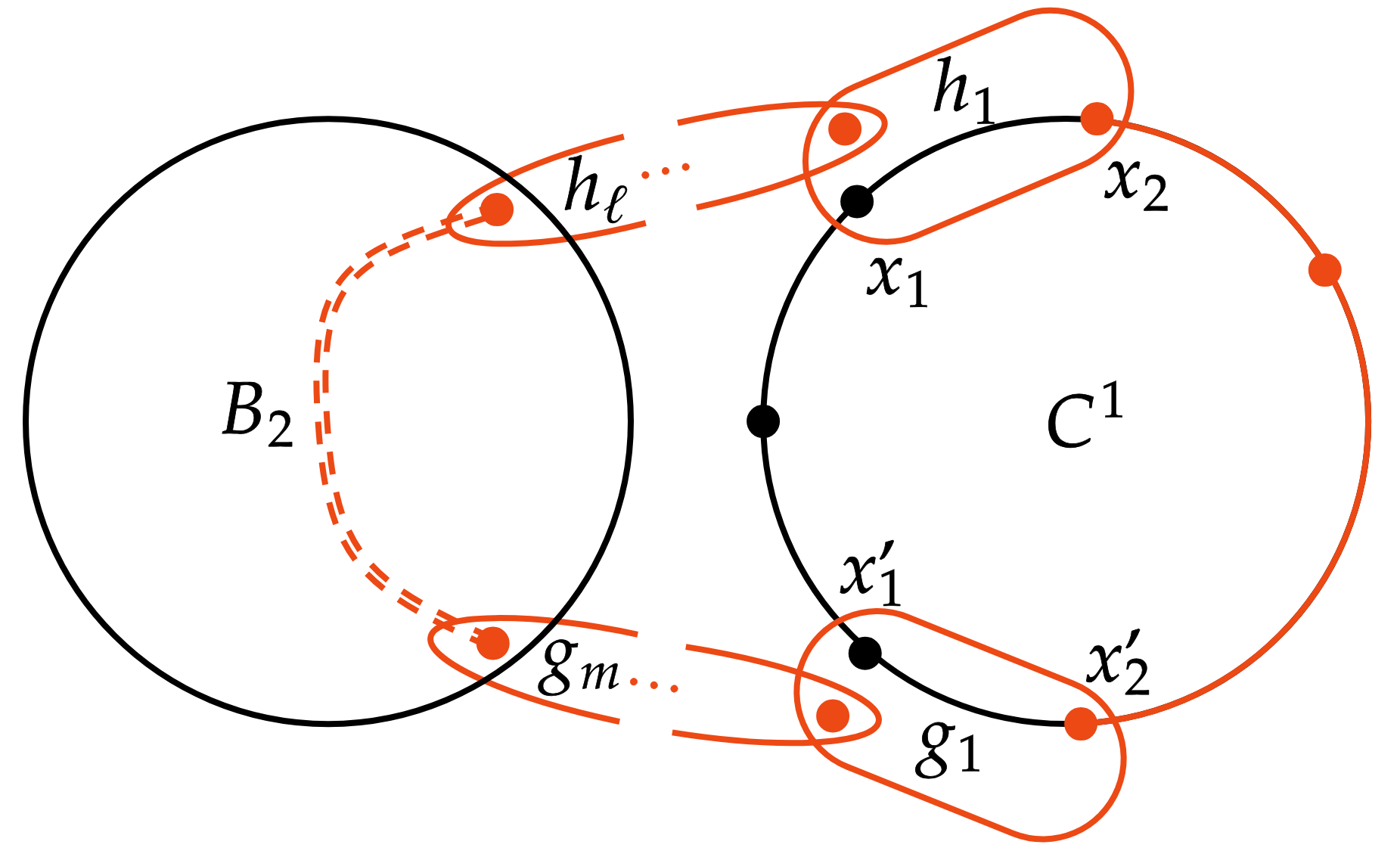}
        \caption{
    The cycle obtained in Situation 3}\label{fig: P12}
    \end{subfigure}
    \hspace{0.01\textwidth}
    \begin{subfigure}[b]{0.42\textwidth}
        \includegraphics[width=\textwidth]{ 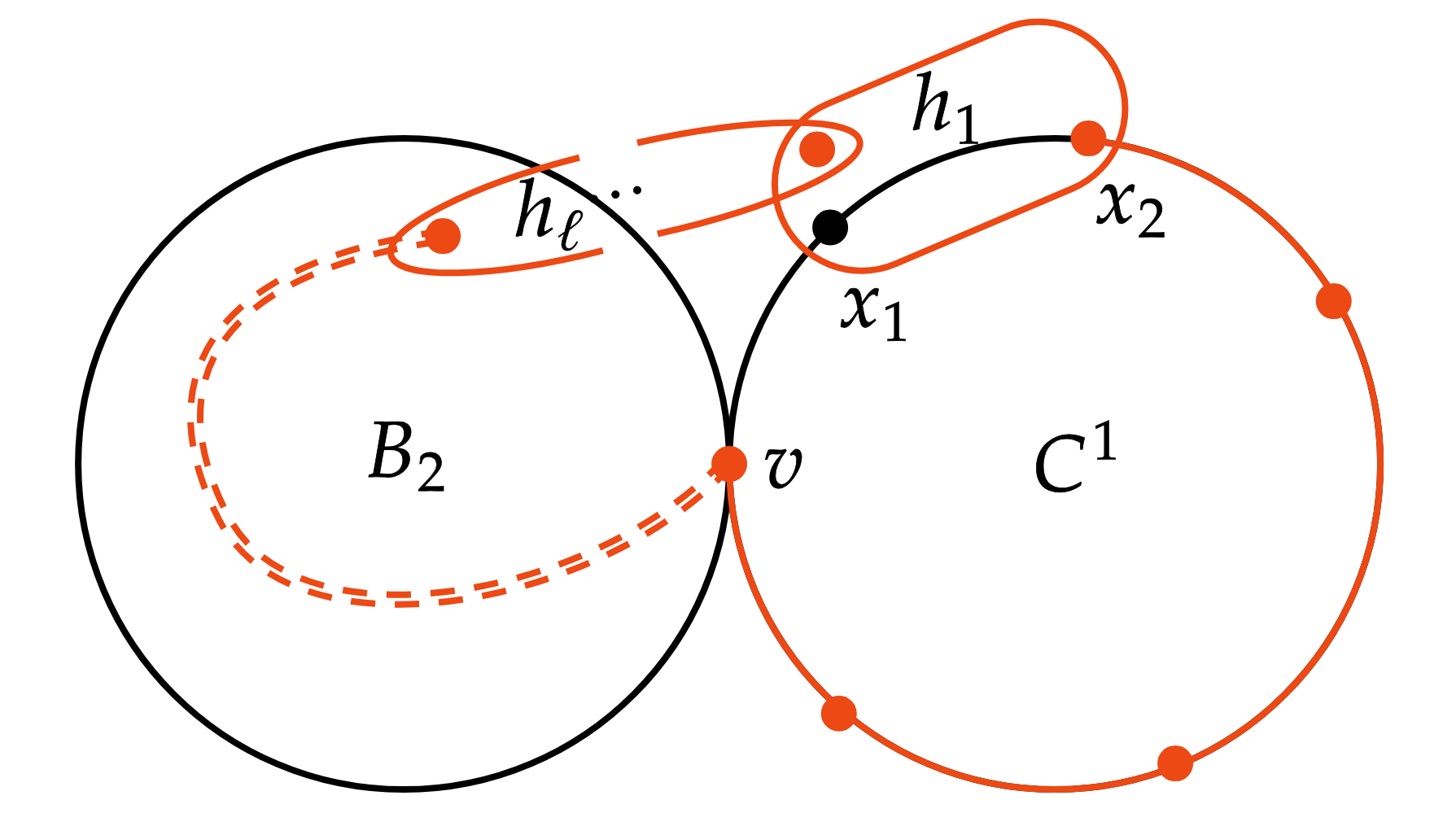}
        \caption{
        The cycle obtained in Situation 1 when $C^1$ and $B_2$ share a cut vertex}\label{fig: P11 share}
    \end{subfigure}
    \end{figure}

\smallskip
\noindent
\textbf{Subcase 1.2.} $C^1$ and $B_2$ share a common cut vertex. 
\smallskip

In this case, we may assume $d_1=d_2=v$ is the common cut vertex, and there is a shortest Berge path $h_1,\dots,h_\ell$ connecting $C^1$ and $B_2$ with end defining vertices $c_1\in C^1$ and $c_2\in B_2$, that does not use $v$ as a defining vertex.
Then the only defining hyperedge intersecting with $V(C^1)$ is $h_1$, and the only defining hyperedge intersecting with $V(B_2)$ is $h_\ell$.
Furthermore, there is a path $P^2$ with at least $k/2+1$ vertices connecting $c_2,v$ in $B_2$ by Claim \ref{clm: nice block find path}, and there is no edge $e_2$ in $P^2$ such that $M(e_2)=h_\ell$ by Claim \ref{clm: no such edge}.

We claim that we can find a path with at least $k/2$ vertices in $C^1$ connecting $v$ and a different vertex in $h_1$ that does not use any edge $e_1$ with $M(e_1)=h_1$. Informally, we can say that $h=h_1$ and there is no $g$, thus we have Situation 1 or 2. More precisely, there is no $e_1$ in $C^1$ such that $M(e_1)=h_1$, then one of the two subpaths connecting $v$ and $c_1$ satisfies these properties, while if there is such an $e_1=x_1x_2$, then either the path connecting $c,x_1$ or the path connecting $c,x_2$ satisfies these properties.

Then the hyperedges $M(e)$ for the edges $e$ in this path and $P^2$, together with $h_1,\dots, h_\ell$ form a
Berge cycle of length at least $k/2+1+k/2-1=k$, a contradiction. Notice that in this case, the vertex $v$ is counted twice, in the paths 
inside $C^1$
and $P^2$.
(See Figure \ref{fig: P11 share} for an example when $C^1$ has situation 2).

This completes the proof of the case where $B_2$ is a nice or strong leaf block.

\smallskip
\noindent
\textbf{Case 2.} $B_2$ is also a troublesome block.
\smallskip

Then, there is a cycle of length at least $k-2$ in $B_2$, which is denoted by $C^2$. 
Then, we consider the cases whether $C^1$ and $C^2$ share a common cut vertex.

\smallskip
\noindent\textbf{Subcase 2.1.} $C^1$ and $C^2$ share no common cut vertex.
\smallskip

Then, we similarly find the shortest Berge path $h_1,\dots,h_\ell$ and $g_1,\dots,g_m$ as above.
Again, $c_1,c_2$ (resp.$d_1,d_2$) are the end defining vertices of $h_1,\dots,h_\ell$ (resp. $g_1,\dots,g_m$) in $C^1$ and $C^2$, respectively.


Recall that in $C^2$ we also have one of the three situations described earlier. In particular, if we find paths with at least $k$ vertices (Situations 1 and 2) in both cycles, then the hyperedges $M(e)$ for the edges of these paths together with $h_1,\dots, h_\ell$ and $g_1,\dots,g_m$ form a Berge cycle with at least $k$ vertices, a contradiction. 

If $\ell\ge 2$ and $m\ge 2$, then the paths inside $C^1$ and $C^2$ have at least $k-4$ edges, the hyperedges $M(e)$ for these edges and the at least 4 hyperedges $h_1,\dots, h_\ell$ and $g_1,\dots,g_m$ form a Berge cycle of length at least $k$, a contradiction. 

If $m=1$ and $\ell\ge 2$, then we cannot have both $C^1$ and $C^2$ in Situation 3, since then an edge $e_1$ inside $C^1$ and an edge $e_2$ inside $C^2$ would have $M(e_1)=M(e_2)=h_1$, which is impossible. Therefore, the paths inside $C^1$ and $C^2$ have at least $k-3$ edges, the hyperedges $M(e)$ for these edges and the at least 3 hyperedges $h_1,\dots, h_\ell$ and $g_1,\dots,g_m$ form a Berge cycle of length at least $k$, a contradiction. The same holds if $m\ge 2$ and $\ell=1$.

Finally, assume that $m=\ell=1$. Similarly to the above, we are done unless one of the cycles, say $C^2$ is in Situation 3. Then $C^1$ is in Situation 1. If the cycle $C^1$ in $B_1$ has length at least $k-1$, then one can easily check that there is a path with at least $k/2+1$ vertices from $c_1$ to $d_1$. If there is no cycle of length at least $k-1$ in $B_1$, then we can apply Lemma \ref{lem: when k=8} to show that there exists a path connecting $c_1,d_1$ with at least $k/2+1$ vertices. 
Together with the path with at least $k/2-1$ vertices in $ B_2$, we can find a Berge cycle of length at least $k$ as in the above cases.

\smallskip
\noindent\textbf{Subcase 2.2.} $C^1$ and $C^2$ share a common cut vertex
\smallskip

In this case, we may assume $d_1=d_2=v$ is the common cut vertex, and there is a shortest Berge path $h_1,\dots,h_\ell$ connecting $C^1$ and $C^2$ with end vertices $c_i\in C^i$ respectively.

Similarly to Subcase 1.2, we can find a subpath with at least $k/2$ vertices in $C^1$ from $c_1$ to $v$ without using $e_1$ with $M(e_1)=h_1$, and a subpath with at least $k/2$ vertices in $C^2$ from $c_2$ to $v$ without using $e_2$ with $M(e_2)=h_\ell$. Note that we only have $k-1$ vertices in the union of these two paths.

If $\ell\ge 2$, then the hyperedges $M(e)$ for the edges $e$ of these two paths together with $h_1,\dots, h_\ell$ form a Berge cycle of length at least $k$, a contradiction. 

If $\ell=1$, then at least one of the two blocks, without loss of generality $B_1$ does not have an edge $e_1$ with $M(e_1)=h_1$. If there is a cycle of length $k-1$ in $B_1$ that contains $c_1$, then there is a path of length at least $k/2+1$ connecting $v$ and $c_1$ on this cycle.  If there is a cycle of length $k-1$ in $B_1$ that does not contain $c_1$, then there is a shortest path inside $B_1$ from $c_1$ to a vertex $c_1'$ this cycle. If $c_1'\neq v$, then there is a  path of length at least $k/2+1$ connecting $v$ and $c_1'$ on this cycle, thus there is a  path of length at least $k/2+1$ connecting $v$ and $c_1$ inside $B_1$. If $c_1'=v$, then since $B_1$ is 2-connected, there is another path from $c_1$ to the cycle of length $k-1$, to a vertex $c_1''$. Then there is a  path of length at least $k/2+1$ connecting $v$ and $c_1''$ on this cycle, thus there is a  path of length at least $k/2+1$ connecting $v$ and $c_1$ inside $B_1$.

If $B_1$ contains no cycles of length at least $k-1$, then we can apply Lemma \ref{lem: when k=8} to show that there exists a path connecting $c_1,v$ with at least $k/2+1$ vertices. 
Together with the path with at least $k/2$ vertices in $ B_2$, we can find a Berge cycle of length at least $k$ as in the above cases.




Now, we have completed the proof of the case where $B_2$ is also troublesome, thereby proving the claim. 
\end{proof}


If there is a troublesome block $B$,
then, according to Claim \ref{clm: no troublesome}, in the deleting process, there is no other leaf block that is nice or strong or troublesome.
It implies that there is an order to delete the bad leaf blocks such that $B$ is the last one.
Since $B$ has size at least $N_{r,k}$, by Corollary \ref{cor: extremal structure} the number of red $r$-cliques and blue edges in $B$ is at most ${\lc (k+1)/2 \rc\choose r}+(|V(B)|-\lc (k+1)/2 \rc)\binom{\lf (k-1)/2\rf }{r-1}$. We can similarly calculate that
\begin{align*}
    g_r(G^{rb})\leq & {\lc (k+1)/2 \rc\choose r}+(|V(B)|-\lc (k+1)/2 \rc)\binom{\lf (k-1)/2\rf }{r-1}+(n-|V(B)|)\binom{\lf (k-1)/2 \rf}{r-1}\\
    = &\binom{ k/2 +1}{r}+(n- k/2 -1)\binom{k/2 -1}{r-1}.
\end{align*}
This completes the proof in this case. Thus, we may assume that there is no troublesome block.
There are two types of bad blocks left in $G$. We deal with them separately. 
\begin{itemize}
    \item \textbf{Type 1:}
     Bad blocks $B$ of order at most $N_{r,k}$, where $C_{r,k}$ is the constant defined in the definition of troublesome blocks. 
     
     During the deletion of vertices of degree at most $ k/2 -1$, for the last vertex we deleted, we only deleted one edge. Thus, according to Claim \ref{clm: rb with one vertex}, the value of $g_r(G^{rb})$ decreases by at most $\binom{ k/2-1}{r-1}(|V(B)|-2)+1\leq \left(\binom{k/2-1}{r-1}-\frac{\binom{k/2-1}{r-1}-1}{C_{r,k}-1} \right)(|V(B)|-1).$
    \item \textbf{Type 2:}
Bad blocks $B$ of order more than $C_{r,k}\geq 4k$, with total number of edges at most $( k/2-\frac{4}{3})|V(B)|$.

During the deleting process, every vertex we deleted has degree at most $ k/2-1$.
The number of vertices we deleted with degree at most $ k/2-2$ is at least $\frac{|V(B)|}{k}$,
otherwise, the number of edges in $B$ is at least $\frac{|V(B)|}{2k}+\lf\frac{2k-1}{2k}(|V(B)|-1)\rf(k/2-1)> (k/2-\frac{4}{3})|V(B)|$ when $|V(B)|\geq C_{r,k}$, a contradiction.
Thus according to Claim \ref{clm: rb with one vertex}, the value of $g_r(G^{rb})$ decreases by at most $$\binom{ k/2-1}{r-1}(|V(B)|-1)-\frac{|V(B)|}{2k}< \left(\binom{k/2-1}{r-1}-\frac{1}{3k} \right)(|V(B)|-1).$$
\end{itemize}

Then, when deleting the Type 1 and Type 2 bad blocks one by one, there is a constant $\delta=\delta(r,k)$ such that by average, for each vertex we delete, the value of $g_r(G^{rb})$ decreases by at most $\binom{k/2-1}{r-1}-\delta$. 

when $n$ is sufficiently large, we have 
\begin{align*}
    g_r(G^{rb})&\leq \binom{N_{r,k}}{r}+\binom{N_{r,k}}{2}+(n-n')\left(\binom{\lf k/2 \rf-1}{r-1}-\delta \right)\\
    &<\binom{k/2+1}{r}+(n-k/2-1)\binom{ k/2-1}{r-1},
\end{align*}
a contradiction. This completes the proof. 
\end{proof}

\section{Concluding remarks}
In this paper, we determined the maximum number of hyperedges in an connected $n$-vertex $r$-uniform Berge-$P_k$-free hypergraph for every $k\geq 2r+2\geq 8$ and sufficiently large $n$.
We also determined the maximum number of hyperedges in a $2$-connected $n$-vertex $r$-uniform hypergraph without Berge cycles of length at least $k$ for every $k\geq 2r+2\geq 8$ and sufficiently large $n$.
Both thresholds on $k$ are the best possible in the sense that the extremal structure we constructed is not optimal for smaller $k$.
A natural question is to determine the maximum number of hyperedges in a connected $n$-vertex $r$-uniform Berge-$P_k$-free hypergraph for every $k< 2r+2$ and sufficiently large $n$, and also for forbidden Berge cycles.

Recall that when $k\geq 4$ and $m$ is sufficiently large, then $W(m,k,\lf k/2\rf-1)$ is the extremal structure for connected $m$-vertex graphs without $k$-vertex path.
For integer $r\geq 3$, we add $r-2$ new vertices to each edge of $W(m,k,\lf k/2\rf-1)$, and the resulting hypergraph is denoted by $\cG(n,k,r)$, where $n=m+(r-2)e(W(m,k,\lf (k-1)/2\rf))$. Then $e(\cG(n,k,r))=e(W(m,k,\lf (k-1)/2\rf))$ and $\cG(n,k,r)$ is a connected $n$-vertex $r$-uniform Berge-$P_k$-free hypergraph.
This implies that for every $k\geq 4$ and $r\geq 3$, $\ex_r^{conn}(n,\text{Berge-}P_k)=\Theta(n)$.
It is natural to ask whether the maximum number of hyperedges in a $2$-connected $n$-vertex $r$-uniform Berge-$\cC_{\geq k}$-free hypergraph is also $\Theta(n)$ for every $k\geq 5$ and $r\geq 3$.


Assume that a Berge-$P_k$-free connected hypergraph $\cH$ has $\ex^{conn}_r(n,\textup{Berge-}P_k)-O(1)$ hyperedges. In our proof, for each vertex not in the nice or strong component, when we remove that vertex, we remove at most $\binom{\lfloor (k-1)/2\rfloor}{r-1}-1$ red $r$-cliques and blue edges. Therefore, when applying the deletion process for $\cH$, we remove $O(1)$ vertices this way. In other words, all but $O(1)$ vertices of $\cH$ are in the same component of the red-blue graph given by Proposition \ref{prop: matching}. It would be interesting to turn this argument into a proper stability result, describing the structure of $\cH$. We remark that in the case $k$ is odd, a similar statement holds when forbidding each Berge cycle of length at least $k$.

\bigskip
\smallskip
\noindent{\bf{Funding}}
\smallskip

The research of Zhao is supported by the China Scholarship Council (No. 202506210250) and
the National Natural Science Foundation of China (Grant 12571372).

The research of Gerbner is supported by the National Research, Development and Innovation Office - NKFIH under the grant KKP-133819 and by the János Bolyai scholarship.

The research of Zhou is supported by the National Natural Science Foundation of China (Nos. 12271337 and 12371347).

\bigskip
\noindent{\bf{Declaration of interest}}
\smallskip

The authors declare no known conflicts of interest.

\end{document}